\documentclass[12pt,a4paper,leqno]{amsart}

\title[A. p. pseudodifferential operators and Gevrey classes]{Almost periodic pseudodifferential operators and Gevrey classes}

\author[A. Oliaro]{Alessandro Oliaro}

\address{Dipartimento di Matematica, Universit{\`a} di Torino, Via Carlo Alberto 10, 10123 Torino (TO), Italy.}

\email{alessandro.oliaro@unito.it}

\author[L. Rodino]{Luigi Rodino}

\address{Dipartimento di Matematica, Universit{\`a} di Torino, Via Carlo Alberto 10, 10123 Torino (TO), Italy.}

\email{luigi.rodino@unito.it}

\author[P. Wahlberg]{Patrik Wahlberg}

\address{Dipartimento di Matematica, Universit{\`a} di Torino, Via Carlo Alberto 10, 10123 Torino (TO), Italy.}

\email{patrik.wahlberg@unito.it}

\usepackage{latexsym}
\usepackage{amsmath}
\usepackage{amssymb}
\usepackage{amsthm}
\usepackage{amsfonts}
\usepackage{mathrsfs}
\usepackage{calc}
\usepackage{cite}

\numberwithin{equation}{section}          

\newtheorem{thm}{Theorem}
\numberwithin{thm}{section}

\newcommand{\rubrik}{}
\newtheorem{prop}[thm]{Proposition}
\newtheorem{cor}[thm]{Corollary}
\newtheorem{lem}[thm]{Lemma}

\theoremstyle{definition}

\newtheorem{defn}[thm]{Definition}
\newtheorem{example}[thm]{Example}

\theoremstyle{remark}

\newtheorem{rem}[thm]{Remark}              

\newcommand{\pd}[1] {\partial ^#1}
\newcommand{\pdd}[2] {\partial_{#1} ^#2}

\newcommand{\ro}{\mathbb R}
\newcommand{\no}{\mathbb N}
\newcommand{\rr}[1]{\mathbb R^{#1}}
\newcommand{\rrb}[1]{\mathbb R_B^{#1}}
\newcommand{\nn}[1]{\mathbb N^{#1}}
\newcommand{\zz}[1]{\mathbb Z^{#1}}

\newcommand{\co}{\mathbb C}
\newcommand{\cc}[1]{\mathbb C^{#1}}

\newcommand{\ep}{\varepsilon}
\newcommand{\de}{\delta}
\newcommand{\la}{\lambda}
\newcommand{\La}{\Lambda}

\newcommand{\supp}{\operatorname{supp}}

\newcommand{\eabs}[1]{\langle #1\rangle}

\newcommand{\vrum}{\vspace{0.1cm}}

\newcommand{\wt}{\widetilde}
\newcommand{\wh}{\widehat}

\begin{document}

\begin{abstract}
We study almost periodic pseudodifferential operators acting on almost periodic functions $G_{\rm ap}^s(\rr d)$ of Gevrey regularity index $s \geq 1$. We prove that almost periodic operators with symbols of H\"ormander type $S_{\rho,\delta}^m$ satisfying an $s$-Gevrey condition are continuous on $G_{\rm ap}^s(\rr d)$ provided $0 < \rho \leq 1$, $\delta=0$ and $s \rho \geq 1$. A calculus is developed for symbols and operators using a notion of regularizing operator adapted to almost periodic Gevrey functions and its duality. We apply the results to show a regularity result in this context for a class of hypoelliptic operators.
\end{abstract}


\maketitle

\let\thefootnote\relax\footnotetext{Corresponding author: Alessandro Oliaro,
phone +390116702863, fax +390116702878. Authors' address: Dipartimento di Matematica, Universit{\`a} di Torino, Via Carlo Alberto 10, 10123 Torino (TO), Italy. Email \{alessandro.oliaro,luigi.rodino,patrik.wahlberg\}@unito.it.}

\section{Introduction}

Almost periodic (a.p.) pseudodifferential operators ($\Psi DO$) on $\rr d$ are operators defined by H\"ormander type symbols $S_{\rho,\delta}^m$ that are almost periodic in the space variables.
They have been studied thoroughly, in particular by M.~A.~Shubin \cite{Shubin1,Shubin2,Shubin3,Shubin4,Shubin4b}, and by
Coburn, Moyer and Singer \cite{Coburn1}, Dedik \cite{Dedik1}, Filippov \cite{Filippov1}, Pankov \cite{Pankov1}, Rabinovich \cite{Rabinovich1} and Wahlberg \cite{Wahlberg1}. The related class of pseudodifferential operators on the torus $\mathbb T^d$ (where the almost periodicity is replaced by periodicity) is treated by Ruzhansky and Turunen \cite{Ruzhansky1} (see also their recent monograph \cite{Ruzhansky2}).

Shubin proved that a.p.~$\Psi DO$s act continuously on smooth a.p. functions, and that the operator norm on $L^2$ equals that on $B^2$ \cite[Theorem 4.4]{Shubin4}. Here $B^2(\rr d)$ is the Hilbert space of Besicovitch a.p. functions such that the Fourier coefficients are square summable. Moreover, Shubin introduced scales of Sobolev spaces $W_t^p(\rr d)$ of a.p. functions, $t \in \ro$, $p \in [1,\infty]$, defined by the norm
$$
\| f \|_{W_t^p} = \left( \sum_{\xi \in \rr d} {\eabs \xi}^{pt} |\wh f_\xi|^p \right)^{1/p},
$$
(for $p \in [1,\infty)$, with obvious modification for $p=\infty$)
where $\wh f_\xi = \mathscr M_x (f(x) e^{- 2 \pi i x \cdot \xi})$ is the Bohr--Fourier transform, $\mathscr M$ being the mean value functional of a.p. functions \cite{Levitan1,Corduneanu1}, and $\eabs{\xi}=(1+|\xi|^2)^{1/2}$.

For $p=2$ and an a.p.~$\Psi DO$ denoted $A$ with symbol in $S_{\rho,\delta}^m$, $0 \leq \delta < \rho \leq 1$, Shubin proved the continuity of $A: W_t^2(\rr d) \mapsto W_{t-m}^2(\rr d)$ for any $t \in \ro$ \cite[Theorem 4.3]{Shubin1}.
Furthermore, in the scale $W_t^2$ and $W_{-\infty}^2=\bigcup_{t \in \ro} W_t^2$, he proved a regularity result for
formally hypoelliptic operators.
More precisely, \cite[Theorem 5.3]{Shubin4} says that if $A$ is a formally hypoelliptic a.p.~$\Psi DO$ with symbol in $S_{\rho,\delta}^{m,m_0}$, $0 \leq \delta < \rho \leq 1$, and $u \in W_{-\infty}^2(\rr d)$ then $A u \in W_t^2(\rr d)$ implies $u \in W_{t+m_0}^{2}(\rr d)$ for any $t \in \ro$.

This paper treats almost periodic pseudodifferential operators acting on a.p. functions that are Gevrey regular of order $s \geq 1$.
The Gevrey scale measures regularity between analytic ($s=1$) and $C^\infty$ (see e.g. \cite{Rodino1}).
We are aiming for Gevrey versions of Shubin's results.
The general $s$-Gevrey functions $f \in G^s(\Omega)$ on an open set $\Omega \subseteq \rr d$ satisfy by definition estimates of the type
\begin{equation}\label{localgevrey1}
\sup_{x \in K} |\pd \alpha f(x)| \leq C_K^{1+|\alpha|} (\alpha!)^s, \quad \alpha \in \nn d, \quad C_K>0,
\end{equation}
for each compact set $K \subseteq \Omega$ \cite{Rodino1}.
A crucial remark, and the first motivation for our paper, is the following.
For a.p.~$s$-Gevrey functions, $s>1$, it is a priori not clear whether a global constant $C>0$ exists such that
\begin{equation}\label{globalgevrey1}
|\pd \alpha f(x)|\leq C^{1+|\alpha|} (\alpha!)^s, \quad \alpha \in \nn d, \quad x \in \rr d.
\end{equation}
We show by means of an example (see Example \ref{counterexample1}) that there exist a.p.~$s$-Gevrey functions such that no finite constant can be used in a global estimate.
After that, we restrict our attention to a.p.~$s$-Gevrey functions that satisfy the global bound \eqref{globalgevrey1}. We introduce the union of such spaces over all $C>0$, denoted $G_{\rm ap}^s(\rr d)$, and equip it with its natural inductive limit topology as a union of Banach spaces.

We define symbol classes such that $a(\cdot,\xi)$ are a.p. functions for all $\xi \in \rr d$, and moreover obey estimates of the form
\begin{equation}\label{gevreysymbol1}
|\pdd x \alpha \pdd \xi \beta a(x,\xi)|
\leq C^{1+|\alpha|+|\beta|} (\alpha!)^{s(\rho-\delta)} \beta! \eabs{\xi}^{m-\rho|\beta|+\delta |\alpha|},
\quad x \in \rr d, \quad \eabs{\xi} \geq B |\beta|^s,
\end{equation}
for multi-indices $\alpha,\beta \in \nn d$, where $C>0$, $B>0$, $m \in \ro$, $0 \leq \delta < \rho \leq 1$, and $s(\rho-\delta) \geq 1$.
A pseudodifferential operator with symbol $a$ acting on a smooth a.p. function $f$ is defined by
\begin{equation}\label{psidodef0}
a(x,D) f(x) = \lim_{\ep \rightarrow +0} \iint_{\rr {2d}} \chi(\ep y)
\chi(\ep\xi) e^{2 \pi i \xi \cdot (x-y)} a(x,\xi) f(y) {\,} dy {\,} d\xi
\end{equation}
where $\chi \in C_c^\infty(\rr d)$ equals one in a neighborhood of the origin.

The symbols defined by \eqref{gevreysymbol1}, without the a.p. condition on $a(\cdot,\xi)$ for all $\xi \in \rr d$, are global versions of symbols for $\Psi DO$s acting on Gevrey spaces by means of the definition
\begin{equation}\label{psidodef1}
a(x,D) f(x) = \int_{\rr d} e^{2 \pi i x \cdot \xi} a(x,\xi) \mathscr F f(\xi) d\xi, \quad f \in G_c^s(\rr d), \quad s>1,
\end{equation}
where $\mathscr F$ denotes the Fourier transform and $G_c^s$ denotes the compactly supported $s$-Gevrey functions.
Many authors have contributed to the calculus of $\Psi DO$s of this form acting on Gevrey spaces, among them
Boutet de Monvel and Kr\'ee \cite{Boutetdemonvel1}, Hashimoto, Matsuzawa and Morimoto \cite{Hashimoto1}, Liess and Rodino \cite{Liess1}, and
Rodino and Zanghirati \cite{Rodino2,Zanghirati1}. For a comprehensive survey we refer to Rodino \cite[Chapter 3]{Rodino1}.
In the above cited literature, the symbol estimates are assumed to hold only locally when $x \in K$ for compact sets $K \subseteq \rr d$, with constants $C=C_K$ and $B=B_K$ depending on $K$, and moreover the polynomial growth term in the symbol $\eabs{\xi}^m$, where $m$ corresponds to a finite order, is relaxed to an exponential bound $C_\ep \exp(\ep |\xi|^{1/s})$ for some $C_\ep>0$, for any $\ep>0$.
This suffices to prove that $a(x,D)$ defined by \eqref{psidodef1} is continuous $a(x,D): G_c^s(\rr d) \mapsto G^s(\rr d)$. By duality $a(x,D)$ extends to a continuous operator from compactly supported $s$-Gevrey ultradistributions to $s$-Gevrey ultradistributions \cite{Rodino1}.

In the following we describe our results where we always assume $\delta=0$.
We show that a symbol that is a.p. in the first variable and satisfying \eqref{gevreysymbol1} gives an operator defined by \eqref{psidodef0} that is continuous on $G_{\rm ap}^s(\rr d)$, for $s \geq 1$. In fact, we show a more general result where the symbol $a(x,\xi)$ in \eqref{psidodef0} is replaced by an amplitude $a(x,y,\xi)$ that depends on both space variables $x,y$ and satisfying estimates corresponding to \eqref{gevreysymbol1}. The continuity admits us to extend the action of $a(x,D)$ to a weak$^*$ continuous map on the topological dual $(G_{\rm ap}^s)'(\rr d)$.
We have the embedding $W_{s,0}^1(\rr d) \subseteq (G_{\rm ap}^s)'(\rr d)$, where $W_{s,0}^1(\rr d)$ denotes the space of a.p. functions $f$ such that
$$
\sum_{\xi \in \rr d} \exp(-\ep |\xi|^{1/s}) |\wh f_\xi| < \infty \quad \forall \ep>0.
$$
It follows that
$$
a(x,D): W_{s,0}^1(\rr d) \mapsto (G_{\rm ap}^s)'(\rr d)
$$
is continuous.
We will adopt $W_{s,0}^1(\rr d)$, rather than $(G_{\rm ap}^s)'(\rr d)$, as a universe set, corresponding to the $s$-Gevrey ultradistributions of compact support in the local calculus \cite{Rodino1}.

In the next step we develop a calculus of a.p.~$\Psi DO$s acting on $G_{\rm ap}^s(\rr d)$. We introduce the notion of a regularizing operator as an operator that is continuous
$$
a(x,D): W_{s,0}^1(\rr d) \mapsto G_{\rm ap}^s(\rr d).
$$
Adapting the calculus of \cite[Chapter 3.2]{Rodino1} (cf. \cite{Hashimoto1,Liess1,Zanghirati1}) to the current situation, we define the concept of a formal sum $\sum_{j \geq 0} a_j$ of symbols and the equivalence relation of formal sums $\sum_{j \geq 0} a_j \sim \sum_{j \geq 0} b_j$, and we prove that $a \sim 0$ implies that $a(x,D)$ is regularizing. We also prove that an a.p.~$\Psi DO$ defined by an amplitude of three variables $a(x,y,\xi)$ can be written as an operator with symbol $b \sim \sum_{j \geq 0} b_j$, where
\begin{equation}\nonumber
b_j(x,\xi) = (2 \pi i)^{-j} \sum_{|\alpha|=j} (\alpha!)^{-1} \pdd y \alpha \pdd \xi \alpha a(x,y,\xi) \big|_{y=x},
\end{equation}
modulo a regularizing operator. The proof is based on ideas from \cite{Rodino1,Hashimoto1,Liess1,Zanghirati1} with some modifications and additional arguments, which take into account almost periodicity (cf. \cite{Shubin1,Shubin2,Shubin3,Shubin4,Shubin4b})
and the uniformness of estimates over the space variables.

Using the the latter result, we prove the following essential part of the calculus. The composition of two a.p.~$\Psi DO$s $a(x,D)$ and $b(x,D)$ can be written as $(a \circ b)(x,D)$ modulo a regularizing operator, where $a \circ b$ is the formal sum $\sum_{j \geq 0} c_j$ where
\begin{equation}\nonumber
c_j = (2 \pi i)^{-j} \sum_{|\alpha|=j} (\alpha!)^{-1} \ \pdd \xi \alpha a \ \pdd x \alpha b.
\end{equation}

As an application of the calculus we show the following result. Let $a$ be an a.p. symbol of type \eqref{gevreysymbol1} that is formally $s$-hypoelliptic in the sense that
there exists $C,C_1>0$, $A,B \geq 0$ and $m_0 \leq m$ such that
\begin{equation}\nonumber
\begin{aligned}
|a(x,\xi)|
\geq & \ C_1 \eabs{\xi}^{m_0}, \quad x \in \rr d, \quad |\xi| \geq A, \\
\left| \left( \pdd x \alpha \pdd \xi \beta a (x,\xi) \right) a(x,\xi)^{-1} \right|
\leq & \ C^{|\alpha|+|\beta|} (\alpha!)^{s\rho} \beta! \eabs{\xi}^{-\rho|\beta|}, \\
& x \in \rr d, \quad |\xi| \geq \max(A,B|\beta|^s),
\end{aligned}
\end{equation}
with $0<\rho \leq 1$, $s \geq 1$, $s \geq 1/\rho$.
Then there exists a symbol $b$ of order $-m_0$ such that $R_1$ and $R_2$ are regularizing, where
$$
R_1 = a(x,D) b(x,D) - I, \quad R_2 = b(x,D) a(x,D) - I.
$$
This gives the following regularity result.
If $a$ is an a.p. formally $s$-hypoelliptic symbol in the above sense, then $f \in W_{s,0}^1(\rr d)$ and $a(x,D)f \in G_{\rm ap}^s(\rr d)$ imply $f \in G_{\rm ap}^s(\rr d)$.

As a consequence we obtain an almost periodic Gevrey version of the celebrated result by H\"ormander \cite{Hormander-1} and Malgrange \cite{Malgrange1}, saying that if a linear partial differential operator $P$ has constant strength and a symbol $p(x,\xi)$ such that $p(x_0,\cdot)$ is a hypoelliptic polynomial for some $x_0 \in \rr d$,
then $P$ is hypoelliptic. An $s$-Gevrey version of this result is \cite[Theorem 3.3.13]{Rodino1}.

Our almost periodic $s$-Gevrey result says that if $P$ is a linear partial differential operator of constant strength with coefficients in $G_{\rm ap}^1(\rr d)$, and a symbol $p$ such that $p(x_0,\cdot)$ is a hypoelliptic polynomial for some $x_0 \in \rr d$ and $|p(x,\xi)| \geq C |p(x_0,\xi)|$ for $C>0$ and $|\xi|$ large and $x \in \rr d$, then $f \in W_{s,0}^1(\rr d)$ and $Pf \in G_{\rm ap}^s(\rr d)$ imply $f \in G_{\rm ap}^s(\rr d)$, for $s \geq 1$.

Let us finally describe how the material is ordered, and add some comments.
Sections \ref{sectionprel} and  \ref{sectionap} are devoted to Gevrey almost periodicity. Almost periodic ultra-distributions were treated by G\'omez--Collado in \cite{GomezCollado1} following the lines of L. Schwartz' almost periodic distributions \cite{Schwartz1}. We follow here a somewhat different approach. First we define $G_{\rm ap}^s (\rr d)$ as the space of almost periodic $s$-Gevrey functions that can be estimated as $|\pd \alpha f(x)|\leq C^{1+|\alpha|}(\alpha!)^s$ with a global constant $C>0$. We also define its topological dual $(G_{\rm ap}^s)' (\rr d)$. The pathologies that one meets when the set of frequencies $\La = \{ \xi \in \rr d: \wh f_\xi \neq 0 \}$ is bounded in $\rr d$ \cite{Shubin1,Shubin4} are not excluded.
In fact, in this general case for $f \in G_{\rm ap}^s (\rr d)$, the bound in Lemma \ref{fouriergevrey1}
$$
| \wh f_\xi | \leq C \exp \left( - \ep |\xi|^{1/s} \right), \quad \ep>0, \quad \xi \in \rr d,
$$
give no more information than boundedness, and we cannot identify $f$ with an element of $W_{s,0}^1 (\rr d)$.
A more comforting situation is recaptured by imposing on $\La$ the (very mild) non-boundedness condition, cf. \eqref{frequencyrequirement1}:
$$
\sum_{\xi \in \La} \exp(- \ep |\xi|^{1/s}) < \infty \quad \forall \ep>0.
$$
Under such an assumption
we can embed $G_{\rm ap}^s (\rr d)$ as a subspace of $W_{s,0}^1(\rr d)$ and the inclusion $G_{\rm ap}^s (\rr d) \subseteq (G_{\rm ap}^s)' (\rr d)$ becomes an embedding, see \eqref{frequencysepinclusion1} in Proposition \ref{frequencyseparation1} and Proposition \ref{dualembedding1}.

Sections \ref{sectionappsdo} and \ref{sectioncalculus} are devoted to Gevrey almost periodic pseudodifferential operators. The main technical difficulty is to keep track when $x \rightarrow \infty$ of the constants in the local estimates of \cite{Hashimoto1,Liess1,Rodino1,Zanghirati1}. We are able to do it in general only for $\delta=0$, see in particular the basic Proposition \ref{gevreycont2} on continuity, though the main definitions and some results are set in the generic classes $S_{\rho,\delta}^m$.

In Section \ref{sectionconstant} we give an application to a.p. partial differential operators of constant strength.

\section{Preliminaries on almost periodic functions}\label{sectionprel}

The topological dual of a topological vector space $X$ is denoted $X'$.
An inclusion $X \subseteq Y$ of two topological spaces $X$ and $Y$ is called an embedding if it is continuous.
The symbol $C(\rr d)$ denotes the space of continuous complex-valued functions and $C_b(\rr d)=C(\rr d) \cap L^\infty (\rr d)$.
We write $D_j=-i \partial/\partial x_j$, $D=(D_1,D_2,\dots, D_d)$, and for $\alpha=(\alpha_1,\dots,\alpha_d) \in \nn d$,
$$
\pdd x \alpha = \frac{\partial^{\alpha_1}}{\partial x^{\alpha_1}} \cdots \frac{\partial^{\alpha_d}}{\partial x^{\alpha_d}}.
$$
For $m \in \no$, the space $C_b^m(\rr d)$ consists of all functions $f$ such that $\pd \alpha f \in C_b(\rr d)$ for all $\alpha \in \nn d$ with $|\alpha| \leq m$, and $C_b^\infty(\rr d) = \bigcap_{m \in \no} C_b^m(\rr d)$. The space of smooth functions with compact support is denoted $C_c^\infty(\rr d)$. The Schwartz space is written $\mathscr S(\rr d)$ and consists of smooth functions such that a derivative of any order multiplied by any polynomial is uniformly bounded. Its topological dual is the space of tempered distributions $\mathscr S'(\rr d)$.
The Fourier transform for $f \in L^1(\rr d)$ is defined by
$$
\mathscr F f(\xi) = \int_{\rr d} f(x) e^{- 2 \pi i \xi \cdot x} dx.
$$
The $L^2$-product $(\cdot,\cdot)_{L^2}$ is conjugate linear in the second argument. For distributions and test functions, e.g. $f \in \mathscr S'(\rr d)$ and $\varphi \in \mathscr S(\rr d)$ we denote by $(f,\varphi)= \langle f, \overline{\varphi} \rangle$ the conjugate linear action of $f$ on $\varphi$. Thus $\langle f,\varphi \rangle$ denotes $f$ acting linearly on $\varphi$.

The space of trigonometric polynomials is denoted $TP(\rr d)$ and consists of finite linear combinations of the form
\begin{equation}\label{trigpol1}
f(x) = \sum_{\xi \in \rr d} a_\xi \ e^{2 \pi i \xi \cdot x}, \quad a_\xi \in \co.
\end{equation}
$TP(\La)$ denotes the subspace of $TP(\rr d)$ where $\xi \in \La$ for $\La \subseteq \rr d$.

The space of uniform almost
periodic functions is denoted $C_{\rm ap}(\rr d)$ and defined as follows. A
set $U \subseteq \rr d$ is called \textit{relatively dense} if there
exists a compact set $K \subseteq \rr d$ such that $(x+K) \cap U \neq
\emptyset$ for any $x \in \rr d$. For $\ep>0$, an element $\tau \in \rr d$ is
called an $\ep$-almost period of a function $f \in C_b(\rr d)$ if
$\sup_x |f(x+\tau)-f(x)|<\ep$. The space $C_{\rm ap}(\rr d)$ is defined as the
set of all $f \in C_b(\rr d)$ such that, for any $\ep>0$, the set
of $\ep$-almost periods of $f$ is relatively dense.
With the understanding that the uniform almost periodic functions is a subspace
of $C_b(\rr d)$, this original definition by H. Bohr is equivalent
to the following three \cite{Levitan1,Shubin4}:

\begin{enumerate}

\item[(i)] The set of translations $\{ f(\cdot -x)\}_{x \in \rr
d}$ is precompact in $C_b(\rr d)$;

\vrum

\item[(ii)] $f=g \circ i_B$, where $i_B$ is the canonical
homomorphism from $\rr d$ into the Bohr compactification $\rrb d$ of
$\rr d$, and $g \in C(\rrb d)$. Hence $f$ can be extended to a
continuous function on $\rrb d$;

\vrum

\item[(iii)] $f$ is the uniform limit of trigonometric
polynomials.

\end{enumerate}

Here the Bohr compactification of a locally compact abelian group $G$ is a compact abelian topological group constructed as the group dual to ($G'$, with discrete topology), that is $G_B=(G'_{\rm discr})'$ \cite{Shubin4}.

The space $C_{\rm ap}(\rr d)$ is a conjugate-invariant complex algebra of
uniformly continuous functions, and a Banach space
with respect to the $L^\infty$ norm.
For $f \in C_{\rm ap}(\rr d)$ the mean
value functional
\begin{equation}\label{meandef1}
\mathscr M(f) = \lim_{T \rightarrow +\infty} T^{-d} \int_{s+K_T} f(x) {\,}
dx,
\end{equation}
where $K_T = \{ x \in \rr d:\ 0 \leq x_j \leq T, \
j=1,\dots,d \}$, exists independent of $s \in \rr d$. By $\mathscr M_x$ we
understand the mean value in the variable $x$ of a function of
several variables. The Bohr--Fourier transformation
\cite{Levitan1} is defined by
\begin{equation}\nonumber
\mathscr F_{B} f (\xi) = \wh f_\xi = \mathscr M_x(f(x) e^{-2 \pi i \xi \cdot x}), \quad \xi \in \rr d, \quad f \in C_{\rm ap}(\rr d),
\end{equation}
and $\wh f_\xi \neq 0$ for at most countably many $\xi \in \rr d$.
The set $\La = \{ \xi \in \rr d: \wh f_\xi \neq 0 \}$ is called the set of frequencies for $f$.
(For $f \in TP(\rr d)$ the frequencies are the $\xi \in \rr d$ such that $a_\xi \neq 0$ in \eqref{trigpol1}.)
A function $f \in C_{\rm ap}(\rr d)$ may be reconstructed from its
Bohr--Fourier coefficients $(\wh f_\xi)_{\xi \in \La}$ using
Bochner--Fej{\' e}r polynomials \cite{Levitan1,Shubin4}.

For $m \in \no$, the space $C_{\rm ap}^m(\rr d)$ is defined as all $f \in
C^m(\rr d)$ such that $\partial^\alpha f \in C_{\rm ap}(\rr d)$ for
$|\alpha| \leq m$, with norm
$$
\| f \|_{C_{\rm ap}^m} = \sum_{|\alpha| \leq m} \sup_{x \in \rr d} | \pd \alpha f(x) |.
$$
We set $C_{\rm ap}^\infty(\rr d) = \bigcap_{m \in \no}
C_{\rm ap}^m(\rr d)$, equipped with the projective limit topology defined by all inclusions
$C_{\rm ap}^\infty(\rr d) \subseteq C_{\rm ap}^m(\rr d)$, $m \in \no$ \cite{Kothe1,Robertson1,Schaefer1}.
By definition, this is the weakest topology such that all inclusions are continuous.
We have $C_{\rm ap}^\infty = C_{\rm ap} \cap C_b^\infty$
\cite{Shubin1}.

The next result is well known and concerns the Bohr--Fourier transformation and differentiation. We include a proof for completeness.

\begin{lem}\label{difffourier1}
If $f \in C_{\rm ap}^\infty(\rr d)$ then $\wh{(\pd \alpha f)}_\xi  = (2 \pi i \xi)^\alpha \wh f_\xi$ for any $\alpha \in \nn d$.
\end{lem}
\begin{proof}
Let $x=(x_1,x')$ where $x_1 \in \ro$, $x' \in \rr {d-1}$ and $K_T' = \{ x' \in \rr {d-1}:\ 0 \leq x_k \leq T, \ k=2,\dots,d \}$. Keeping in mind that $|K_T'|=T^{d-1}$ and $f$ is bounded, integration by parts gives
\begin{equation}\nonumber
\begin{aligned}
\wh{(\partial_1 f)}_\xi = & \lim_{T \rightarrow +\infty} T^{-d} \int_{K_T'} \left( \int_{0}^T \partial_1 f(x_1,x') e^{- 2 \pi i x_1 \xi_1} dx_1 \right) \, e^{- 2 \pi i x' \cdot \xi'} dx' \\
= & \lim_{T \rightarrow +\infty} T^{-d} \left( \int_{K_T'} \left[ f(x_1,x') e^{- 2 \pi i x_1 \xi_1}  \right]_0^T e^{- 2 \pi i x' \cdot \xi'} dx' \right. \\
& \left. + 2 \pi i \xi_1 \int_{K_T}  f(x) e^{- 2 \pi i x \cdot \xi} dx \right) \\
= & \ 2 \pi i \xi_1 \wh f_\xi.
\end{aligned}
\end{equation}
Likewise $\wh{(\partial_j f)}_\xi = 2 \pi i \xi_j \wh f_\xi$ for $1 < j \leq d$ and the result follows by induction.
\end{proof}

The mean value defines a sesquilinear form
\begin{equation}\label{besicovitch1}
(f,g)_B = \mathscr M(f \overline g), \quad f,g \in C_{\rm ap}(\rr d).
\end{equation}
If we set $\| f \|_B = (f,f)_B^{1/2}$ we obtain a norm on $C_{\rm ap}(\rr d)$; in fact, due to Parseval's formula for $C_{\rm ap}(\rr d)$ \cite{Levitan1}
$$
\mathscr M ( |f|^2 ) = \sum_{\xi \in \rr d} | \wh f_\xi |^2,
$$
we have the implication $\| f \|_B = 0$ $\Rightarrow f=0$.
The completion of $TP(\rr d)$, or $C_{\rm ap}(\rr d)$, in the norm $\| \cdot \|_B$ is the
Hilbert space of Besicovitch a.p. functions $B^2(\rr d)$
\cite{Shubin1}. We have the Hilbert space isomorphism $B^2(\rr d) \simeq L^2(\rrb d)$ where $L^2(\rrb d)$ denotes the space of square integrable functions on the Bohr compactification $\rrb d$, equipped with its Haar measure $\mu$, normalized to $\mu(\rrb d)=1$ \cite{Shubin4}.

Inspired by the usual Sobolev space norms for tempered distributions on $\rr d$
\begin{equation}\nonumber
\| f \|_{H_t(\rr d)} = \left(\int_{\rr d} {\eabs \xi}^{2 t} |\widehat
f(\xi)|^2 {\,} d\xi \right)^{1/2}, \quad t \in \ro,
\end{equation}
Shubin \cite{Shubin1} has defined Sobolev--Besicovitch spaces of
a.p. functions $H_t^2(\rr d)$ for $t \in \ro$, as the completion of
$TP(\rr d)$ in the norm corresponding to the inner product
\begin{equation}\nonumber
(f,g)_{H_t^2(\rr d)} = \sum_{\xi \in \rr d}
{\eabs \xi}^{2 t} \wh f_\xi \ \overline{\wh g}_\xi, \quad f,g \in TP(\rr d).
\end{equation}
The spaces $H_t^2(\rr d)$ are nonseparable Hilbert spaces, $H_0^2(\rr d)=B^2(\rr d)$, and one defines
\begin{equation}\nonumber
H_{\infty}^2(\rr d) = \bigcap_{t \in \ro} H_t^2(\rr d), \quad
H_{-\infty}^2(\rr d) = \bigcup_{t \in \ro} H_t^2(\rr d).
\end{equation}
$H_{\infty}^2(\rr d)$ is equipped with the projective limit topology defined by all
inclusions $H_{\infty}^2(\rr d) \subseteq H_t^2(\rr d)$ for $t \in \ro$.
$H_{-\infty}^2(\rr d)$ is equipped with the inductive limit topology \cite{Kothe1,Robertson1,Schaefer1}.
This means that $H_{-\infty}^2(\rr d)$ has the strongest locally convex topology such that the inclusions $H_t^2(\rr d) \subseteq H_{-\infty}^2(\rr d)$, for all $t \in \ro$, are continuous.

We will use the following lemmas which are consequences of general results for projective and inductive limit topologies defined by families of locally convex spaces. For their proofs, see \cite[Theorem II.5.2 and Theorem II.6.1]{Schaefer1}, respectively.

\begin{lem}\label{projectivecont1}
Let $Z$ be a linear space, and let $X_j \subseteq Z$, $j \in J$, be Banach spaces indexed by a partially ordered set $J$.
Define on $X = \bigcap_{j \in J} X_j$ the projective limit topology defined by all inclusions $i_j: X \subseteq X_j$, $j \in J$.
If $Y$ is a topological space and $T$ is a map from $Y$ into $X$, then $T$ is continuous if and only if $i_j \circ T: Y \mapsto X_j$ is continuous for each $j \in J$.
\end{lem}

\begin{lem}\label{inductivecont1}
Let $Z$ be a linear space, and let $X_j \subseteq Z$, $j \in J$, be Banach spaces indexed by a partially ordered set $J$.
Define on $X = \bigcup_{j \in J} X_j$ the inductive limit topology defined by all inclusions $i_j: X_j \subseteq X$, $j \in J$.
If $Y$ is a locally convex space and $T$ is a linear map from $X$ into $Y$, then $T$ is continuous if and only if the restriction $T \circ i_j = T \big|_{X_j} : X_j \mapsto Y$ is continuous for each $j \in J$.
\end{lem}

Lemma \ref{difffourier1} and the embedding $C_{\rm ap}(\rr d) \subseteq B^2(\rr d)$ imply the embedding $C_{\rm ap}^\infty(\rr d) \subseteq H_{\infty}^2(\rr d)$. But there is no result corresponding to the Sobolev embedding
theorem for the Sobolev--Besicovitch spaces. In fact, $H_{\infty}^2(\rr d)$ is not embedded in $C_{\rm ap}(\rr d)$ \cite{Shubin1}.

The notion of Sobolev spaces of a.p. functions can be generalized to the family of Banach spaces $W_t^p(\rr d)$, $p \in [1,\infty]$ and $t \in \ro$, defined by completion of $TP(\rr d)$ in the norm
\begin{equation}\label{besicovitchsobolev1}
\| f \|_{W_t^p} = \left( \sum_{\xi \in \rr d} {\eabs \xi}^{p t} |\wh f_\xi|^p \right)^{1/p},
\end{equation}
when $p \in [1,\infty)$ (with the standard modification when $p=\infty$),
and
\begin{equation}\label{besicovitchsobolev2}
W_{\infty}^p(\rr d) = \bigcap_{t \in \ro} W_t^p(\rr d),
\quad W_{-\infty}^p(\rr d) = \bigcup_{t \in \ro} W_t^p(\rr d),
\end{equation}
equipped with their projective and inductive limit topologies, respectively.
The Sobolev--Besicovitch spaces $W_t^p(\rr d)$ for $p \in [1,\infty]$ and $t \in \ro$ have been studied in \cite{Dellacqua1,Iannacci1,Shubin1}.
An important special case is $W_0^1(\rr d)$ consisting of functions $f$ with finite Bohr--Fourier coefficient $l^1(\rr d)$ norm \cite{Shubin1,Shubin2,Shubin4}. For $f \in W_0^1(\rr d)$ the Fourier series is absolutely convergent, which shows the embedding
\begin{equation}\label{fourieralgebraembedding1}
W_0^1(\rr d) \subseteq C_{\rm ap}(\rr d).
\end{equation}
The inclusion is strict (cf. \cite[Exercise I.6.6]{Katznelson1}). Moreover, $W_0^1(\rr d)$ is a subalgebra of $C_{\rm ap}(\rr d)$ with respect to pointwise multiplication.

We will adopt a different scale of weights on the Bohr--Fourier coefficients.
Namely we define for $s \geq 1$, $p \in [1,\infty]$ and $\ep \in \ro$ the Banach space $W_{s,\ep}^p(\rr d)$ as the completion of $TP(\rr d)$ with respect to the norm
\begin{equation}\label{besicovitchsobolev3}
\| f \|_{W_{s,\ep}^p} = \left( \sum_{\xi \in \rr d} \exp( - p \ \ep |\xi|^{1/s}) |\wh f_\xi|^p \right)^{1/p}
\end{equation}
(modified as usual for $p=\infty$).
In particular we use
\begin{equation}\nonumber
\| f \|_{W_{s,\ep}^1} = \sum_{\xi \in \rr d} \exp( - \ep |\xi|^{1/s}) |\wh f_\xi|.
\end{equation}
When $\ep>0$ the space $W_{s,\ep}^1(\rr d)$ contains elements such that $\wh f_\xi$ grows very rapidly with $\xi$. We set
\begin{equation}\label{besicovitchsobolev4}
W_{s,0}^p(\rr d) = \bigcap_{\ep>0} W_{s,\ep}^p(\rr d),
\end{equation}
and equip the space $W_{s,0}^p(\rr d)$ with the projective limit topology defined by all inclusions $W_{s,0}^p(\rr d) \subseteq W_{s,\ep}^p(\rr d)$ for $\ep>0$. Finally we define
\begin{equation}\label{besicovitchsobolev5}
W_{s,0-}^p(\rr d) = \bigcup_{\ep<0} W_{s,\ep}^p(\rr d),
\end{equation}
and equip $W_{s,0-}^p(\rr d)$ with the inductive limit topology defined by all inclusions $W_{s,\ep}^p(\rr d) \subseteq W_{s,0-}^p(\rr d)$, $\ep<0$.
The embedding
\begin{equation}\nonumber
W_{s,0-}^p(\rr d) \subseteq W_{s,0}^p(\rr d)
\end{equation}
follows from Lemma \ref{projectivecont1} and Lemma \ref{inductivecont1}.

\section{Almost periodic Gevrey regularity}\label{sectionap}

Let $\Omega \subseteq \rr d$ be open and $s \geq 1$.
The Gevrey space of order $s$ is denoted $G^s(\Omega)$ and defined as the space of all $f \in C^\infty(\Omega)$ such that for any compact set $K \subseteq \Omega$ there is a constant $C>0$ such that
\begin{equation}\label{gevreyestimate1}
|\pd \alpha f(x)| \leq C^{1 + |\alpha|} (\alpha!)^s, \quad x \in K, \quad \alpha \in \nn d.
\end{equation}
Then $G^1(\Omega) = A(\Omega)$ is the space of analytic functions in $\Omega$. When $s>1$ the Gevrey space $G^s(\Omega)$ contains functions of compact support;
the set of compactly supported $G^s(\Omega)$ is denoted by $G_c^s(\Omega)$.
Each Gevrey space $G^s(\Omega)$ is an algebra with respect to multiplication, and is invariant under differentiation. For more information about the Gevrey spaces we refer to \cite{Rodino1}.

We shall study functions that are both almost periodic and members of a fixed Gevrey class $G^s(\rr d)$,
$s>1$. A natural question is whether the estimates for the derivatives \eqref{gevreyestimate1} are global when $ f \in (C_{\rm ap} \cap G^s)(\rr d)$. This means that we have
\begin{equation}\label{gevreyestimate1bis}
|\pd \alpha f(x)| \leq B C^{|\alpha|} (\alpha!)^s, \quad x \in \rr d, \quad \alpha \in \nn d,
\end{equation}
for some constants $B,C>0$. The following example shows that this is not the case in general.

\begin{example}\label{counterexample1}
Let $s>1$, $d=1$,
\begin{equation}\label{gevreyexample1}
g_s(x) = \left\{
\begin{array}{ll}
\exp\left( - x^{- \frac1{s-1}} \right), & x>0, \\
0, & x \leq 0,
\end{array}
\right.
\end{equation}
and
$$
\psi(x) = g_s(x) g_s(1-x).
$$
Then $\supp \psi = [0,1]$ and $\psi \in G_c^s(\ro)$ (see for example \cite{Chung1}), and hence there exists $C>0$ such that
$$
\sup_{x \in \ro} |\pd j \psi(x)| \leq C^j (j!)^s, \quad j \geq 0.
$$
Denote
\begin{equation}\nonumber
C_0 = \sup_{x \in \ro, \ j \geq 1} \left( |\pd j \psi(x)| (j!)^{-s} \right)^{1/j} >0
\end{equation}
so that
\begin{equation}\label{smallestconstant1}
\sup_{x \in \ro, \ j \geq 1} |\pd j \psi(x)| (j!)^{-s} C^{-j} \leq 1 \quad \Longleftrightarrow \quad C \geq C_0.
\end{equation}
Next we define $\psi_n(x)=\psi(nx)$ for integer $n \geq 1$, so that $\supp \psi_n \subseteq [0,1/n]$, and we put
\begin{equation}\nonumber
\varphi_n(x) = \sum_{k \in \mathbb Z} \psi_n(x-2^n(1 + 2 k)).
\end{equation}
Then $\varphi_n$ is $2^{n+1}$-periodic and $\varphi_n \in G^s(\ro)$.
Moreover, $\supp \varphi_n \cap \supp \varphi_m = \emptyset$ for $m \neq n$.
We define
\begin{equation}\label{counterexampledefinition1}
f(x) = \sum_{n=1}^\infty n^{-1/4} \varphi_n(x).
\end{equation}
Let $x \in \ro$ be fixed and let $N \geq 1$.
Since $\{ \varphi_n \}_{n \geq 1}$ have pairwise disjoint support, we have
$$
\left| \sum_{n>N} n^{-1/4} \varphi_n(x) \right| = \sum_{n>N} n^{-1/4} \varphi_n(x) = n_0^{-1/4} \varphi_{n_0}(x)
$$
for some $n_0 > N$, depending on $x$.
Therefore the series \eqref{counterexampledefinition1} converges uniformly, and $f \in C_{\rm ap}(\ro)$ since $\varphi_n \in C_{\rm ap}(\ro)$ for all $n \geq 1$.
The pairwise disjoint support of $\{ \varphi_n \}_{n \geq 1}$ gives furthermore $f \in G^s(\ro)$. Thus $f \in (C_{\rm ap} \cap G^s)(\ro)$.

Let $C>0$, let $n$ be an arbitrary integer such that $n>(C/C_0)^2$ and let $0 \leq x \leq 1/n$. Then $\varphi_m(2^n + x)=0$ for $m \neq n$, and $\varphi_n(2^n + x) = \psi_n(x)$,
which gives
\begin{equation}\nonumber
\begin{aligned}
C^{-j} (j!)^{-s} |\pd j f(2^n + x)| & = n^{-1/4} C^{-j} (j!)^{-s} |(\pd j \psi_n) (x)| \\
& = n^{-1/4} \left( C n^{-1} \right)^{-j} (j!)^{-s} |(\pd j \psi) (n x)|.
\end{aligned}
\end{equation}
We have for $j \geq 1$
$$
n^{-1/4} \left( C n^{-1} \right)^{-j} = n^{j/2-1/4} \left( C n^{-1/2} \right)^{-j} \geq n^{1/4} \left( C n^{-1/2} \right)^{-j},
$$
which implies, in view of \eqref{smallestconstant1}, since $C n^{-1/2}<C_0$,
\begin{equation}\nonumber
\begin{aligned}
& \sup_{0 \leq x \leq 1/n, \ j \geq 1} C^{-j} (j!)^{-s} |\pd j f(2^n + x)| \\
& \geq n^{1/4} \sup_{0 \leq x \leq 1/n, \ j \geq 1} \left( C n^{-1/2} \right)^{-j} (j!)^{-s} |(\pd j \psi) (n x)| \\
& > n^{1/4}.
\end{aligned}
\end{equation}
Since $n$ can be arbitrarily large we have
\begin{equation}\nonumber
\sup_{x \in \ro, \ j \geq 1} C^{-j} (j!)^{-s} |\pd j f(x)| = \infty
\end{equation}
for any $C>0$, which means that
\begin{equation}\nonumber
\sup_{x \in \ro} |\pd j f(x)| \leq B C^{j} (j!)^{s}, \quad j \geq 0,
\end{equation}
cannot hold for any $B,C>0$.
\hspace{68mm} $\Box$
\end{example}

In this paper we will restrict our attention to functions $f \in (C_{\rm ap} \cap G^s)(\rr d)$ such that the uniform estimate \begin{equation}\label{gevreyestimate2}
|\pd \alpha f(x)| \leq C^{1+|\alpha|} (\alpha!)^s, \quad x \in \rr d, \quad \alpha \in \nn d,
\end{equation}
holds for some $C>0$. Example \ref{counterexample1} shows that this is a strict subset of $(C_{\rm ap} \cap G^s)(\rr d)$.

This restriction is consistent with the definition of analytic almost periodic functions \cite{Corduneanu1}. In fact, such a function $f$ is required to be analytic in a strip $\rr d \times i I \subseteq \cc d$ where $I = I_1 \times \cdots \times I_d \subseteq \rr d$ and $I_j=[a_j,b_j]$ with $a_j<0<b_j$ for each $1 \leq j \leq d$. The function $f$ is almost periodic if for any $\ep>0$ there exists a compact set $K_\ep \subseteq \rr d$
such that for any $x \in \rr d$ there is an $u \in x + K_\ep$ such that
$$
|f(z+u) - f(z)| < \ep, \quad z \in \rr d \times i I.
$$
Cauchy's integral formula, see for example \cite{Krantz1}, gives the global estimate
\begin{equation}\nonumber
\sup_{x \in \rr d} |\pd \alpha f(x)| \leq C^{1+|\alpha|} \alpha!, \quad \alpha \in \nn d, \quad x \in \rr d,
\end{equation}
for some $C>0$.

We denote the space of $f \in (C_{\rm ap} \cap G^s)(\rr d)$ that satisfies \eqref{gevreyestimate2} for $C>0$ by $G_{{\rm ap},C}^s(\rr d)$, for $s \geq 1$. It is equipped with the norm
\begin{equation}\label{banachnorm1}
\| f \|_{s,C} = \sup_{\alpha \in \nn d}  C^{-|\alpha|} (\alpha!)^{-s} \sup_{x \in \rr d} |\pd \alpha f(x)|.
\end{equation}
The space of functions such that $\| f \|_{s,C}<\infty$ is a Banach space.
Suppose $C>0$ and $(f_n) \subseteq G_{{\rm ap},C}^s(\rr d)$ is a Cauchy sequence with pointwise limit $f$.
Since $\| \pd \alpha (f - f_n) \|_{L^\infty} \leq C^{|\alpha|} (\alpha!)^{s} \| f - f_n \|_{s,C}$ for each $\alpha \in \nn d$, and the $C_{\rm ap}(\rr d)$ property is preserved under uniform convergence \cite{Levitan1}, the limit $f$ belongs to $C_{\rm ap}^\infty (\rr d)$. Hence $G_{{\rm ap},C}^s(\rr d)$ is a Banach space. Clearly we have the embedding $G_{{\rm ap},C_1}^s(\rr d) \subseteq G_{{\rm ap},C_2}^s(\rr d)$ for $C_2 \geq C_1$.

For any $\xi \in \rr d$ and $e_\xi (x)=e^{2 \pi i \xi \cdot x}$,
the inequality $t^n \leq n! e^t$ for $t \geq 0$ and $n \in \no$
yields for any $s \geq 1$ and any $C>0$
\begin{equation}\label{exponentialgevreybound0}
\begin{aligned}
\left| \pd \alpha e_\xi (x) \right| \leq | 2 \pi \xi|^{|\alpha|}
& = C^{|\alpha|} (\alpha!)^s \left( \frac{\left| 2 \pi \xi C^{-1} \right|^{|\alpha|/s} }{\alpha!} \right)^s \\
& \leq C^{|\alpha|} (\alpha!)^s \exp \left( s d \left| 2 \pi \xi C^{-1} \right|^{1/s} \right).
\end{aligned}
\end{equation}
Hence $TP(\rr d) \subseteq G_{{\rm ap}, C}^s(\rr d)$ for any $C>0$ and any $s \geq 1$, and moreover we have
$\| e_\xi \|_{s,C} \leq \exp \left( s d (2 \pi)^{1/s} C^{-1/s} |\xi|^{1/s} \right)$ for any $C>0$.

\begin{rem}
For $C_2 \geq C_1$, the embedding $G_{{\rm ap},C_1}^s(\rr d) \subseteq G_{{\rm ap},C_2}^s(\rr d)$ is not compact.
In fact, let $(\xi_n)_0^\infty \subseteq \rr d$ be a sequence of distinct frequency vectors bounded as $|\xi_n| \leq C_1$ for all $n \geq 0$. It follows from \eqref{exponentialgevreybound0} that $\sup_{n \geq 0} \| e_{\xi_n} \|_{s,C_1} < \infty$.
Suppose $(e_{\xi_n})_{n=0}^\infty$ has a subsequence $(e_{\xi_{n_k}})_{k=0}^\infty$ that converges in $G_{{\rm ap},C_2}^s(\rr d)$. For $k,j$ sufficiently large and $k \neq j$ we then obtain the contradiction
$$
1 > \| e_{\xi_{n_k}} - e_{\xi_{n_j}} \|_{s,C_2} \geq \sup_{x \in \rr d} | e_{\xi_{n_k}}(x) - e_{\xi_{n_j}}(x) | = 2,
$$
which proves that no subsequence of $(\xi_n)_0^\infty \subseteq \rr d$ converges in $G_{{\rm ap},C_2}^s(\rr d)$.
It follows that the embedding is not compact.
\end{rem}

We set
\begin{equation}\label{spaceunion1}
G_{\rm ap}^s(\rr d) = \bigcup_{C>0} G_{{\rm ap},C}^s(\rr d)
\end{equation}
and equip $G_{\rm ap}^s(\rr d)$ with the inductive limit topology defined by the inclusions $G_{{\rm ap},C}^s \subseteq G_{\rm ap}^s$ for all $C>0$.
However, the inductive limit topology on $G_{\rm ap}^s (\rr d)$ is not a \emph{strict} inductive limit topology \cite{Kothe1,Schaefer1,Reed1,Robertson1}. Therefore we do not know whether $G_{\rm ap}^s$ is complete, or even Hausdorff, for the inductive limit topology, nor do we know whether sequential convergence in $G_{\rm ap}^s$ implies sequential convergence in $G_{{\rm ap},C}^s$ for some $C>0$.

The spaces $G_{\rm ap}^s(\rr d)$ are embedded as
$$
G_{\rm ap}^1(\rr d) \subseteq G_{\rm ap}^s(\rr d) \subseteq G_{\rm ap}^t(\rr d), \quad s \leq t.
$$
The spaces $\{ G_{\rm ap}^s(\rr d) \}_{s \geq 1}$ are pairwise distinct, as can be seen by taking any element $f \in G_c^t(\rr d) \setminus G_c^s(\rr d)$ for $s<t$ (cf. \cite{Chung1}) and performing a periodization.

For any $s \geq 1$, and $C>0$ and any $m \in \no$ we have the embedding
$$
G_{{\rm ap},C}^s (\rr d) \subseteq C_{\rm ap}^m (\rr d).
$$
By Lemma \ref{projectivecont1} and Lemma \ref{inductivecont1} we thus have the embedding $G_{\rm ap}^s \subseteq C_{\rm ap}^\infty$.

Let $f, g \in G_{\rm ap}^s(\rr d)$, which means that $f \in G_{{\rm ap},C_f}^s(\rr d)$ and $g \in G_{{\rm ap},C_g}^s(\rr d)$ for some $C_f,C_g>0$. Leibniz' formula gives
\begin{equation}\nonumber
\begin{aligned}
|\pd \alpha (fg)(x)| & \leq \sum_{\beta \leq \alpha} \binom \alpha \beta |\pd \beta f(x)| |\partial^{\alpha-\beta} g(x)| \\
& \leq \sum_{\beta \leq \alpha} \frac{(\alpha!)^s}{((\alpha-\beta)! \beta!)^s} C_f^{1+|\beta|} C_g^{1+|\alpha-\beta|} (\beta! (\alpha-\beta)!)^{s} \\
& \leq C^{2(1+|\alpha|)} (\alpha!)^s
\end{aligned}
\end{equation}
where $C=\max(2,C_f,C_g)$. The fact that $C_{\rm ap}(\rr d)$ is an algebra under multiplication (see e.g. \cite{Levitan1}) hence shows that $G_{\rm ap}^s(\rr d)$ is an algebra under multiplication. Moreover, it can be shown that $G_{\rm ap}^s(\rr d)$ is an algebra under mean value convolution, defined by
$$
f *_{\small \mathscr M} g(x) = \mathscr M_y( f(x-y) g(y) ).
$$
For this operation we also have the Bohr--Fourier coefficient relation
$$
\wh { (f *_{\small \mathscr M} g ) }_\xi = \wh f_\xi \ \wh g_\xi, \quad \xi \in \rr d.
$$
Concerning the Bohr--Fourier transform acting on $G_{\rm ap}^s(\rr d)$ we have the following result.

\begin{lem}\label{fouriergevrey1}
If $s \geq 1$, $C>0$, $f \in G_{{\rm ap},C}^s(\rr d)$ and
\begin{equation}\label{epsilondefinition1}
\ep = \frac{s}{e} \left( \frac{2 \pi}{C d^{1/2}} \right)^{1/s}
\end{equation}
then there exists $C_1>0$ such that
\begin{equation}\nonumber
|\wh f_\xi| \leq C_1 \| f \|_{s,C}  \exp(- \ep |\xi|^{1/s}), \quad \xi \in \rr d.
\end{equation}
\end{lem}
\begin{proof}
Lemma \ref{difffourier1} yields for any $\alpha \in \nn d$
\begin{equation}\nonumber
|(2 \pi i \xi)^\alpha \wh f_\xi| = | \wh{(\pd \alpha f})_\xi|
\leq \| \pd \alpha f \|_{L^\infty}
\leq \| f \|_{s,C} \ C^{|\alpha|}(\alpha!)^s.
\end{equation}
Using
$$
|\xi|^k \leq d^{k/2} \max_{|\alpha|=k} |\xi^\alpha|,
$$
we obtain the estimate
\begin{equation}\nonumber
|\xi|^k |\wh f_\xi|
\leq \| f \|_{s,C} \left( \frac{C d^{1/2}}{2 \pi} \right)^k k^{k s}, \quad k=1,2,\dots.
\end{equation}
Finally an application of \cite[Chapter IV.2.1]{Gelfand1} yields the result.
\end{proof}

For any $t \in \ro$ and any $\ep>0$ there exists $C_{\ep,t}>0$ such that
\begin{equation}\label{polyexpinequality1}
\exp(-\ep |\xi|^{1/s}) \leq C_{\ep,t} \eabs{\xi}^t, \quad \xi \in \rr d.
\end{equation}
Therefore we have the embedding $W_t^p \subseteq W_{s,\ep}^p$ for all $t \in \ro$, $p \in [1,\infty]$, which
in view of Lemma \ref{projectivecont1} and Lemma \ref{inductivecont1}
implies the embedding
\begin{equation}\label{besicovitchsobolevembedding1}
W_{-\infty}^p (\rr d) \subseteq W_{s,0}^p (\rr d), \quad p \in [1,\infty].
\end{equation}
This indicates that $W_{s,0}^1$ is a rather large space.
However, although we have the embeddings
$$
G_{\rm ap}^s \subseteq C_{\rm ap}^\infty \subseteq W_{\infty}^2 \subseteq W_{-\infty}^2 \subseteq W_{s,0}^2
$$
we cannot prove the inclusion $G_{\rm ap}^s(\rr d) \subseteq W_{s,0}^1(\rr d)$ in general.

Nevertheless, under certain assumptions on the set of frequencies $\La = \{ \xi \in \rr d: \ \wh f_\xi \neq 0 \}$ such an inclusion is possible. Let $C>0$ and let $\La \subseteq \rr d$ be an arbitrary countable set. We define the Banach space
$$
G_{{\rm ap},C,\La}^s(\rr d) = \left\{ f \in G_{{\rm ap},C}^s(\rr d): \ \wh f_\xi = 0, \ \xi \notin \La \right\},
$$
and the inductive limit topological vector space
$$
G_{{\rm ap},\La}^s(\rr d) = \bigcup_{C>0} G_{{\rm ap},C,\La}^s(\rr d).
$$
Likewise we define $W_{t,\La}^p$ for $t \in \ro$ as in \eqref{besicovitchsobolev1} with the restriction that
$\wh f_\xi=0$ for $\xi \notin \La$,
and similarly we define $W_{\infty,\La}^p$, $W_{-\infty,\La}^p$, $W_{s,\ep,\La}^p$ for $\ep \in \ro$, $W_{s,0,\La}^p$ and $W_{s,0-,\La}^p$. See \eqref{besicovitchsobolev2}, \eqref{besicovitchsobolev3}, \eqref{besicovitchsobolev4} and \eqref{besicovitchsobolev5}.

Suppose that
\begin{equation}\label{frequencyrequirement1}
\sum_{\xi \in \La} \exp(- \ep |\xi|^{1/s}) < \infty \quad \forall \ep>0.
\end{equation}
This holds for example if $\La=\zz d$, or more generally if $\sum_{\xi \in \La} \eabs{\xi}^{-t} < \infty$ for some $t>0$.

\begin{prop}\label{frequencyseparation1}
If \eqref{frequencyrequirement1} holds then we have the embeddings
\begin{equation}\label{frequencysepinclusion1}
G_{{\rm ap},\La}^s (\rr d) \subseteq W_{\infty,\La}^p (\rr d) \subseteq W_{-\infty,\La}^p (\rr d) \subseteq W_{s,0,\La}^p (\rr d), \quad p \in [1,\infty),
\end{equation}
and the homeomorphism
\begin{equation}\label{frequencysepequality1}
G_{{\rm ap},\La}^s (\rr d) = W_{s,0-,\La}^p(\rr d), \quad p \in [1,\infty).
\end{equation}
\end{prop}

\begin{proof}
Note in particular that \eqref{frequencysepequality1} implies that $W_{s,0-,\La}^p(\rr d)$ does not depend on $p \in [1,\infty)$.

First we prove \eqref{frequencysepinclusion1}.
Let $p \in [1,\infty)$, $C>0$, $f \in G_{{\rm ap},C,\La}^s (\rr d)$ and $t \in \ro$. From Lemma \ref{fouriergevrey1} and \eqref{polyexpinequality1} we obtain for some $C_1,\ep>0$
\begin{equation}\nonumber
\begin{aligned}
\left( \sum_{\xi \in \La} \eabs{\xi}^{p t} | \wh f_\xi|^p \right)^{1/p}
& \leq C_1 \| f \|_{s,C} \left( \sum_{\xi \in \La} \eabs{\xi}^{p t} \exp(-2 \ep p |\xi|^{1/s}) \right)^{1/p} \\
& \leq C_1 C_{\ep,t} \| f \|_{s,C} \left( \sum_{\xi \in \La} \exp(-\ep p|\xi|^{1/s}) \right)^{1/p}.
\end{aligned}
\end{equation}
Thus \eqref{frequencyrequirement1} implies the embedding $G_{{\rm ap},C,\La}^s \subseteq W_{t,\La}^p$ for any $t \in \ro$ and $p \in [1,\infty)$. Lemma \ref{projectivecont1} and Lemma \ref{inductivecont1} give the first embedding of \eqref{frequencysepinclusion1}. The second embedding is evident, and
the final embedding follows from \eqref{besicovitchsobolevembedding1}.

Next we prove \eqref{frequencysepequality1}. Let $C>0$ and $f \in G_{{\rm ap},C,\La}^s(\rr d)$.
Define $\ep>0$ by \eqref{epsilondefinition1} and let $p \in [1,\infty)$.
Lemma \ref{fouriergevrey1} and \eqref{frequencyrequirement1} yield
\begin{equation}\nonumber
\begin{aligned}
\| f \|_{W_{s,-\ep/2,\La}^p} & = \left( \sum_{\xi \in \La} \exp( p \ep/2 \ |\xi|^{1/s}) |\wh f_\xi|^p \right)^{1/p} \\
& \leq C_1 \| f \|_{s,C} \left( \sum_{\xi \in \La} \exp(-p \ep/2 \ |\xi|^{1/s})\right)^{1/p}
\leq C_2 \| f \|_{s,C}
\end{aligned}
\end{equation}
for some $C_1,C_2>0$. Thus we have the embedding
$$
G_{{\rm ap},C,\La}^s(\rr d) \subseteq W_{s,0-,\La}^p (\rr d).
$$
Since $C>0$ is arbitrary, Lemma \ref{inductivecont1} proves the embedding ``$\subseteq$'' in \eqref{frequencysepequality1}. For the opposite embedding, let $f \in TP(\La)$, that is,
$$
f(x) = \sum_{\xi \in \La} \wh f_\xi \ e^{2 \pi i x \cdot \xi}
$$
with a finite sum. Let $\ep>0$ and pick $C \geq 2 \pi (2 d s/\ep)^s$.
Writing $\theta := (C/(2\pi))^{1/s} \ep/2$, we have $\theta \geq d s$, and for $\alpha \in \nn d$ and $\xi \in \rr d$
\begin{equation}\nonumber
\begin{aligned}
\left( \frac{2 \pi |\xi|}{C} \right)^{|\alpha|} (\alpha !)^{-s}
& = \left( \frac{ds}{\theta} \right)^{s |\alpha|} \left( \frac{( \theta/(ds) (2 \pi/C)^{1/s} |\xi|^{1/s} )^{|\alpha|}}{\alpha!} \right)^s \\
& \leq \exp \left( \theta (2 \pi /C)^{1/s} |\xi|^{1/s} \right) = \exp \left( \ep/2 |\xi|^{1/s} \right).
\end{aligned}
\end{equation}
This gives
\begin{equation}\nonumber
\begin{aligned}
\| f \|_{s,C} & = \sup_{\alpha \in \nn d, \ x \in \rr d} C^{-|\alpha|} (\alpha!)^{-s} |\pd \alpha f(x)|
\leq \sup_{\alpha \in \nn d} \sum_{\xi \in \La} \left( \frac{2 \pi |\xi|}{C} \right)^{|\alpha|} (\alpha!)^{-s} |\wh f_\xi| \\
& \leq \sum_{\xi \in \La} \exp \left( \ep \ |\xi|^{1/s} \right) |\wh f_\xi| \exp \left( -\ep/2 \ |\xi|^{1/s} \right)
\leq C_1 \| f \|_{W_{s,-\ep,\La}^p}
\end{aligned}
\end{equation}
for some $C_1>0$, in the last step invoking H\"older's inequality and \eqref{frequencyrequirement1}.
Since $TP(\La)$ is dense in $W_{s,-\ep,\La}^p(\rr d)$, we may conclude that the embedding $W_{s,-\ep,\La}^p(\rr d) \subseteq G_{{\rm ap},\La}^s(\rr d)$ holds true.
Finally the embedding ``$\supseteq$'' in \eqref{frequencysepequality1} follows from Lemma \ref{inductivecont1}.
\end{proof}

\subsection{The dual space of $G_{{\rm ap}}^s(\rr d)$}

The topological dual of $G_{{\rm ap}}^s(\rr d)$ is denoted $(G_{{\rm ap}}^s)'(\rr d)$. We equip $(G_{{\rm ap}}^s)'(\rr d)$ with its weak$^*$ topology.
We have the following embedding result.

\begin{prop}\label{dualembedding1}
Let $f \in W_{s,0}^1(\rr d)$ define a conjugate linear functional on $G_{\rm ap}^s(\rr d)$ by means of
\begin{equation}\nonumber
( f ,g ) = \sum_{\xi \in \rr d} \wh f_\xi \ \overline{\wh g_\xi}, \quad g \in G_{\rm ap}^s(\rr d).
\end{equation}
Then we have the embedding
$$
W_{s,0}^1(\rr d) \subseteq (G_{{\rm ap}}^s)'(\rr d).
$$
\end{prop}
\begin{proof}
Let $f \in W_{s,0}^1(\rr d)$, $C>0$ and $g \in G_{{\rm ap},C}^s(\rr d)$.
Lemma \ref{fouriergevrey1} gives for some $C_1,\ep>0$
\begin{equation}\nonumber
|( f ,g ) | \leq C_1 \| g \|_{s,C} \sum_{\xi \in \rr d} |\wh f_\xi| \exp\left( - \ep |\xi|^{1/s} \right)
= C_1 \| g \|_{s,C} \| f \|_{W_{s,\ep}^1},
\end{equation}
which shows the embedding $W_{s,\ep}^1(\rr d) \subseteq (G_{{\rm ap},C}^s)'(\rr d)$ when $(G_{{\rm ap},C}^s)'(\rr d)$ carries its weak$^*$ topology.
Hence we have the embedding $W_{s,0}^1(\rr d) \subseteq (G_{{\rm ap},C}^s)'(\rr d)$ for any $C>0$.
It follows from Lemma \ref{projectivecont1} that we have the embedding
\begin{equation}\nonumber
W_{s,0}^1(\rr d) \subseteq \bigcap_{C>0} (G_{{\rm ap},C}^s)'(\rr d)
\end{equation}
where the right hand side has the projective limit topology.

Lemma \ref{inductivecont1} gives the equality as sets
\begin{equation}\label{projlimit1}
\bigcap_{C>0} (G_{{\rm ap},C}^s )'(\rr d) = (G_{\rm ap}^s)'(\rr d).
\end{equation}
To prove the claimed embedding, it thus suffices to prove that \eqref{projlimit1} is a homeomorphism,
when $(G_{\rm ap}^s)'(\rr d)$ carries the weak$^*$ topology,
and the left hand side is the projective limit of $(G_{{\rm ap},C}^s )'(\rr d)$ over all $C>0$,
where $(G_{{\rm ap},C}^s )'(\rr d)$ carries its weak$^*$ topology.

A base of the neighborhoods of zero of $(G_{\rm ap}^s)'(\rr d)$ are given by all finite intersections of the form
\begin{equation}\label{neighborhood1}
\bigcap_{j=1}^n \{ f \in (G_{{\rm ap}}^s )'(\rr d): \ |( f, g_j )|<\ep_j \}, \quad g_j \in G_{{\rm ap}}^s (\rr d), \ \ep_j>0, \ 1 \leq j \leq n.
\end{equation}
On the other hand, for $C>0$, a base of the neighborhoods of zero of $(G_{{\rm ap},C}^s )'(\rr d)$ consists of all finite intersections of the form
\begin{equation}\label{neighborhood2}
\begin{aligned}
U_C = & \ \bigcap_{j=1}^k \{ f \in (G_{{\rm ap},C}^s )'(\rr d): \ |( f, g_j )|<\ep_j \}, \\
& \quad g_j \in G_{{\rm ap},C}^s (\rr d), \ \ep_j>0, \ 1 \leq j \leq k.
\end{aligned}
\end{equation}
Moreover, a base of the neighborhoods of zero of the left hand side projective limit in \eqref{projlimit1} consists of all finite intersections of the form
\begin{equation}\label{neighborhood3}
\bigcap_{C>0} (G_{{\rm ap},C}^s )'(\rr d) \cap U_{C_1} \cap \cdots \cap U_{C_m}
\end{equation}
where $U_{C_j}$ have the form \eqref{neighborhood2}, $C_j>0$, $1 \leq j \leq m$ (see \cite[Chapter II.5]{Schaefer1}).
It follows from the already established set equality \eqref{projlimit1} that the neighborhood bases of the forms \eqref{neighborhood3} and \eqref{neighborhood1} respectively, are equal, proving the homeomorphism \eqref{projlimit1}. (See \cite[Proposition IV.4.5]{Schaefer1} for a similar result when the inductive limit, corresponding to $G_{{\rm ap}}^s(\rr d)$, is known to be Hausdorff.)
\end{proof}

We may identify $G_{{\rm ap}}^s(\rr d)$ as a subspace of $(G_{{\rm ap}}^s)'(\rr d)$ by means of
$f \in G_{{\rm ap}}^s(\rr d)$ acting on $g \in G_{{\rm ap}}^s(\rr d)$ via $(f,g)_B$.
The inclusion $G_{{\rm ap}}^s(\rr d) \subseteq (G_{{\rm ap}}^s)'(\rr d)$ is injective, since $(f,g)_B=0$ for all $g \in G_{{\rm ap}}^s(\rr d)$ forces $\wh f _\xi=0$ for all $\xi \in \rr d$.
In the case when \eqref{frequencyrequirement1} is satisfied for the set of frequencies $\La$,
Propositions \ref{frequencyseparation1} and \ref{dualembedding1} imply that
the inclusion is an embedding.

Note also that, independently of \eqref{frequencyrequirement1}, we may identify $W_{s,0}^1(\rr d)$ for $s>1$ with a subspace of the space of $s$-ultradistributions $\mathscr D_s'(\rr d)$ \cite{Rodino1}. In fact, for $f \in W_{s,0}^1(\rr d)$ and $\varphi \in G_c^s(\rr d)$ we may define
\begin{equation}\nonumber
\langle f, \varphi \rangle = \sum_{\xi \in \rr d} \wh f_\xi \ \mathscr F \varphi (\xi),
\end{equation}
cf. \cite[Theorem 1.6.1]{Rodino1}.

\section{Almost periodic pseudodifferential operators acting on almost periodic Gevrey spaces}\label{sectionappsdo}

First we discuss pseudodifferential operators acting on $C_b^\infty(\rr d)$ (see\cite{Shubin1,Shubin4}).

\begin{defn}\label{symbclass0}
Let $m \in \ro$, $0 \leq \rho,\delta \leq 1$ and
\begin{equation}\label{rhodeltakrav1}
0 < \rho \leq 1, \quad 0 \leq \delta < 1, \quad \delta
\leq \rho.
\end{equation}
The H\"ormander symbol class $S_{\rho,\delta}^{m}$ is defined as the space of all $a \in C^\infty(\rr {2d})$ such that
\begin{equation}\label{symbclass0b}
|\partial_\xi^\alpha \partial_x^\beta a(x,\xi)| \leq C_{\alpha,\beta} \eabs \xi^{m - \rho |\alpha| + \delta |\beta|}, \quad \alpha,\beta \in \nn d, \quad C_{\alpha,\beta}>0.
\end{equation}
\end{defn}

In this paper we always assume that \eqref{rhodeltakrav1} is satisfied.
For $a \in S_{\rho,\delta}^{m}$, a pseudodifferential operator acting on $\mathscr S(\rr d)$ in the Kohn--Nirenberg quantization is defined as (see e.g. \cite{Folland1,Hormander1,Shubin5})
\begin{equation}\label{kndef2}
a(x,D) f(x) = \int_{\rr d} e^{2 \pi i \xi \cdot x} a(x,\xi) \mathscr F
f(\xi) {\,} d\xi, \quad f \in \mathscr S(\rr d).
\end{equation}
This gives a continuous operator on $\mathscr S (\rr d)$, that can be extended to a continuous operator on $\mathscr S' (\rr d)$ by means of duality.
However, in this paper we are interested in the action of pseudodifferential operators on $C_b^\infty(\rr d)$, and in this situation it is suitable to define them as the limit (cf. \cite{Shubin5})
\begin{equation}\label{extension2}
a(x,D) f(x) = \lim_{\ep \rightarrow +0} \iint_{\rr {2d}} \chi(\ep y)
\chi(\ep\xi) e^{2 \pi i \xi \cdot (x-y)} a(x,\xi) f(y) {\,} dy {\,}
d\xi
\end{equation}
where $\chi \in C_c^\infty(\rr d)$ equals one in a neighborhood of
the origin.
Integration by parts in \eqref{extension2} gives
\begin{equation}\label{psdopartint1}
\begin{aligned}
a(x,D) f(x) = & \iint_{\rr {2d}} e^{2 \pi i \xi \cdot (x-y)}
\eabs{\xi}^{-2N} (1-\Delta_\xi)^M a(x,\xi) \\
& \times (1-\Delta_y)^N( \eabs{x-y}^{-2M} f(y) ) {\,} dy {\,} d\xi \\
= & \iint_{\rr {2d}} e^{-2 \pi i \xi \cdot y}
\eabs{\xi}^{-2N} (1-\Delta_\xi)^M a(x,\xi) \\
& \times (1-\Delta_y)^N( \eabs{y}^{-2M} f(y+x) ) {\,} dy {\,} d\xi, \\
\end{aligned}
\end{equation}
where $\Delta$ denotes the normalized Laplacian $\Delta = (2
\pi)^{-2} \sum_1^d \partial_j^2$. This is an absolutely convergent
integral for $f \in C_b^\infty(\rr d)$ provided $M,N$ are positive integers chosen such that $2M>d$ and $2N>d+m$.
By differentiation under the integral it follows that $a(x,D):
C_b^\infty(\rr d) \mapsto C_b^\infty(\rr d)$ continuously (see e.g. \cite{Shubin4}).

The subspace of $S_{\rho,\delta}^m$ consisting of symbols $a \in S_{\rho,\delta}^m$ that satisfy $a(\cdot,\xi) \in C_{\rm ap}(\rr d)$ for all $\xi \in \rr d$ is denoted $APS_{\rho,\delta}^m (\rr {2d})$ and was introduced by Shubin \cite{Shubin1}. The corresponding operators are called almost periodic pseudodifferential operators.
A result by Shubin \cite{Shubin2,Shubin4} says that $a(x,D)$ is continuous on $C_{\rm ap}^\infty(\rr d)$ if $a \in APS_{\rho,\delta}^m (\rr {2d})$.

Now we introduce symbol classes
$$
APS_{\rho,\delta}^{m,s} (\rr {2d}) \subseteq APS_{\rho,\delta}^m (\rr {2d})
$$
based on \cite{Hashimoto1,Liess1,Zanghirati1} and \cite[Chapter 3]{Rodino1} for almost periodic pseudodifferential operators acting on almost periodic Gevrey functions.

\begin{defn}\label{symbclass1}
Let $s \geq 1$, $\rho>\delta$, $s(\rho-\delta) \geq 1$. The symbol class $APS_{\rho,\delta}^{m,s}(\rr {2d})$ is defined as the space of all $a \in C^\infty (\rr {2d})$ such that
$a(\cdot,\xi) \in C_{\rm ap}(\rr d)$ for all $\xi \in \rr d$, and
\begin{align}
|\partial_\xi^\alpha \partial_x^\beta a(x,\xi)|
\leq & \ C^{1+ |\alpha| + |\beta|} (\beta!)^{s(\rho-\delta)} \alpha! \eabs \xi^{m - \rho |\alpha| + \delta |\beta|}, \label{symbclass2} \\
& x \in \rr d, \quad \eabs \xi \geq B |\alpha|^s, \quad \alpha,\beta \in \nn d, \nonumber \\
|\partial_\xi^\alpha \partial_x^\beta a(x,\xi)|
\leq & \ C^{1+ |\alpha| + |\beta|} (\beta!)^{s(\rho-\delta)} \alpha!, \label{symbclass2p} \\
& x \in \rr d, \quad \eabs \xi < B |\alpha|^s, \quad \alpha,\beta \in \nn d, \quad |\alpha| \leq 2M, \nonumber
\end{align}
for some constants $C,B>0$ and a positive integer $M>d/2$.
\end{defn}

The reader may take as initial definition of symbols the estimates \eqref{symbclass2} in the whole $\rr {2d}$, which imply \eqref{symbclass2p}; however, the more general Definition \ref{symbclass1} will be useful when constructing parametrices.

We also need a generalization of the symbol classes of Definition \ref{symbclass1}, in the sense that the symbols depend on three variables, in which case they are called amplitudes.

\begin{defn}\label{symbclass3}
Let $s \geq 1$, $\rho>\delta$, $s(\rho-\delta) \geq 1$. The amplitude class $APS_{\rho,\delta}^{m,s}(\rr {3d})$ is defined as the space of all $a \in C^\infty (\rr {3d})$ such that
$a(\cdot,\cdot,\xi) \in C_{\rm ap}(\rr {2d})$ for all $\xi \in \rr d$, and
\begin{align}
|\partial_\xi^\alpha \partial_x^\beta \partial_y^\gamma a(x,y,\xi)|
\leq & \ C^{1+ |\alpha| + |\beta| + |\gamma|} (\beta! \gamma!)^{s(\rho-\delta)} \alpha! \eabs \xi^{m - \rho |\alpha| + \delta (|\beta|+|\gamma|)}, \label{symbclass3a} \\
& x,y \in \rr d, \quad \eabs \xi \geq B |\alpha|^s, \quad \alpha,\beta,\gamma \in \nn d, \nonumber \\
|\partial_\xi^\alpha \partial_x^\beta \partial_y^\gamma a(x,y,\xi)|
\leq & \ C^{1+ |\alpha| + |\beta| + |\gamma|} (\beta! \gamma!)^{s(\rho-\delta)} \alpha!, \label{symbclass3ap} \\
& x,y \in \rr d, \quad \eabs \xi < B |\alpha|^s \quad \alpha,\beta,\gamma \in \nn d, \quad |\alpha| \leq 2M, \nonumber
\end{align}
for some constants $C,B>0$ and a positive integer $M>d/2$.
\end{defn}

An amplitude $a \in APS_{\rho,\delta}^{m,s}(\rr {3d})$ defines an operator acting on $f \in C_b^\infty(\rr d)$ by means of the limit
\begin{equation}\label{extension3}
a(x,D) f(x) = \lim_{\ep \rightarrow +0} \iint_{\rr {2d}} \chi(\ep y)
\chi(\ep\xi) e^{2 \pi i \xi \cdot (x-y)} a(x,y,\xi) f(y) {\,} dy {\,}
d\xi.
\end{equation}
Integration by parts gives
\begin{equation}\label{psdopartint3}
\begin{aligned}
a(x,D) f(x)
= &  \iint_{\rr {2d}} e^{2 \pi i \xi \cdot (x-y)}
\eabs{\xi}^{-2N}  \\
& \times (1-\Delta_y)^N \left[ \eabs{x-y}^{-2M} f(y) (1-\Delta_\xi)^M a(x,y,\xi) \right] {\,} dy {\,} d\xi \\
= &  \iint_{\rr {2d}} e^{-2 \pi i \xi \cdot y}
\eabs{\xi}^{-2N}  \\
& \times (1-\Delta_y)^N \left[ \eabs{y}^{-2M} f(y+x) (1-\Delta_\xi)^M a(x,x+y,\xi) \right] {\,} dy {\,} d\xi.
\end{aligned}
\end{equation}
The integral is absolutely convergent for $f \in C_b^\infty(\rr d)$ provided $2M>d$ and $2N>(d+m)/(1-\delta)$.
By differentiation under the integral it follows that $a(x,D):
C_b^\infty(\rr d) \mapsto C_b^\infty(\rr d)$ continuously.
Moreover, from Shubin's argument for the continuity on $C_{\rm ap}^\infty(\rr d)$ (cf. \cite{Shubin4}) it follows that $a(x,D): C_{\rm ap}^\infty(\rr d) \mapsto C_{\rm ap}^\infty(\rr d)$ when $a \in APS_{\rho,\delta}^{m,s}(\rr {3d})$.
In fact, using the Banach space property of $C_{\rm ap}(\rr d)$ (equipped with the $L^\infty$-norm), a derivative of the left hand side of \eqref{psdopartint3} can be seen as a vector-valued integral, of a continuous function with values in $C_{\rm ap}(\rr d)$, whose norm is integrable. It therefore belongs to $C_{\rm ap}(\rr d)$ and this gives the desired continuity on $C_{\rm ap}^\infty(\rr d)$.

In \cite{Rodino1} results are proved concerning $s$-Gevrey pseudolocality, which means that an operator preserves the $G^s$ property locally. The following result concerns a global version of this idea for almost periodic operators acting on almost periodic functions. Note that we require $\delta=0$.

\begin{prop}\label{gevreycont2}
If $a \in APS_{\rho,0}^{m,s}(\rr {3d})$ then $a(x,D)$ is continuous on $G_{\rm ap}^s(\rr d)$.
\end{prop}
\begin{proof}
We may assume that the constant $C>0$ in \eqref{symbclass3a}, \eqref{symbclass3ap} satisfies $C \geq 1$.
From \cite[Proposition 3.1]{Shubin4}, and as discussed above, we have that $a(x,D)$ maps $C_{\rm ap}^\infty(\rr d)$ into itself continuously.
It remains to prove continuity on $G_{\rm ap}^s(\rr d)$.
Let $C_1>0$, $f \in G_{{\rm ap},C_1}^s(\rr d)$ and $2N>d+m$.
First we rewrite \eqref{psdopartint3} as
\begin{equation}\nonumber
\begin{aligned}
a(x,D) f(x)
= &  \iint_{\rr {2d}} e^{-2 \pi i \xi \cdot y}
\eabs{\xi}^{-2N}  \\
& \sum_{\stackrel{|\sigma|+|\kappa|+|\mu| \leq 2N}{|\gamma| \leq 2M}} C_{\sigma,\kappa,\mu,\gamma} \
\pd \sigma \eabs{y}^{-2M} \pdd y \kappa f(y+x) \pdd y \mu \pdd \xi \gamma a(x,x+y,\xi) {\,} dy {\,} d\xi.
\end{aligned}
\end{equation}
We may differentiate under the integral since the differentiated integrand is absolutely integrable. Hence we obtain, using \eqref{symbclass3a}, \eqref{symbclass3ap} and \eqref{banachnorm1},
\begin{equation}\label{psdopartint4}
\begin{aligned}
& |\pd \alpha \left( a(x,D) f(x) \right) |
\leq \sum_{\stackrel{|\sigma|+|\kappa|+|\mu| \leq 2N}{|\gamma| \leq 2M}} |C_{\sigma,\kappa,\mu,\gamma}|
\sum_{\beta \leq \alpha} \binom \alpha \beta
\sum_{\la \leq \alpha-\beta} \binom {\alpha-\beta} {\lambda} \\
& \times \iint_{\rr {2d}}
\eabs{\xi}^{-2N} \left| \pd \sigma \eabs{y}^{-2M} \partial_y^{\kappa+\beta} f(y+x) 
\left( \partial_x^{\alpha-\beta-\lambda} \partial_y^{\mu+\lambda} \pdd \xi \gamma a \right) (x,x+y,\xi) \right| {\,} dy {\,} d\xi \\
\leq & \ C_{N,M} \| f \|_{s,C_1}
\sum_{\stackrel{|\sigma|+|\kappa|+|\mu| \leq 2N}{|\gamma| \leq 2M, \ \beta \leq \alpha, \ \la \leq \alpha-\beta}}
\binom \alpha \beta \binom {\alpha-\beta} {\lambda}
C_1^{|\kappa|+|\beta|} \left((\kappa+\beta)! \right)^s \\
& \times C^{1+|\alpha|-|\beta|+|\mu|+|\gamma|} ((\alpha-\beta-\la)! (\mu+\la)!)^{s \rho} \gamma! \\
\leq & \ D_{N,M} \| f \|_{s,C_1} C_3^{1 + |\alpha|} (\alpha!)^s
\sum_{\stackrel{|\sigma|+|\kappa|+|\mu| \leq 2N}{|\gamma| \leq 2M, \ \beta \leq \alpha, \ \la \leq \alpha-\beta}}
\frac{\left((\kappa+\beta)! (\alpha-\beta-\la)! (\mu+\la)! \right)^s}{(\beta! \la! (\alpha-\beta-\la)!)^s},
\end{aligned}
\end{equation}
where $C_{N,M}, D_{N,M}>0$ depend on $N,M$ only and $C_3 = \max(C,C_1)$.
In fact, for $\eabs{\xi} \geq B|\gamma|^s$ we may use \eqref{symbclass3a}. The remaining integral over
$\{\xi: \eabs{\xi} < B|\gamma|^s\}$ can be estimated by \eqref{symbclass3ap} times the Lebesgue measure of
$\{\xi: \eabs{\xi} < B(2 M)^s\}$.
Next we use
\begin{equation}\nonumber
\begin{aligned}
(\kappa+\beta)! \leq 2^{|\kappa|+|\beta|} \kappa! \beta ! \quad \mbox{and} \quad \sum_{\beta \leq \alpha} \binom{\alpha}{\beta} = 2^{|\alpha|},
\end{aligned}
\end{equation}
which give
\begin{equation}\nonumber
\begin{aligned}
& \sum_{\stackrel{|\sigma|+|\kappa|+|\mu| \leq 2N}{|\gamma| \leq 2M, \ \beta \leq \alpha, \ \la \leq \alpha-\beta}}
\frac{\left((\kappa+\beta)! (\la+\mu)! \right)^s}{(\beta! \la!)^s}
\leq \sum_{\stackrel{|\sigma|+|\kappa|+|\mu| \leq 2N}{|\gamma| \leq 2M, \ \beta \leq \alpha, \ \la \leq \alpha-\beta}}
2^{s(|\kappa|+|\beta|+|\la|+|\mu|)} (\kappa! \mu!)^s \\
& \leq E_{N,M} 2^{2 s|\alpha|} \sum_{\beta \leq \alpha} \binom{\alpha}{\beta} \sum_{\la \leq \alpha-\beta} \binom{\alpha-\beta}{\la}
= E_{N,M} 2^{2 (s+1) |\alpha|},
\end{aligned}
\end{equation}
where $E_{N,M}>0$ depends on $N,M$ only. Inserted into \eqref{psdopartint4} this gives
\begin{equation}\nonumber
|\pd \alpha \left( a(x,D) f(x) \right) | \leq \ D_{N,M} E_{N,M} \| f \|_{s,C_1} C_2^{1 + |\alpha|} (\alpha!)^s,\quad x \in \rr d,
\end{equation}
where $C_2 = \max(C,C_1,2^{2(s+1)})^2$.
This gives the estimate
\begin{equation}\label{continuityestimate1}
\| a(x,D) f \|_{s,C_2} \leq C_2 D_{N,M} E_{N,M} \| f \|_{s,C_1},
\end{equation}
which proves the continuity of $a(x,D): G_{{\rm ap},C_1}^s(\rr d) \mapsto G_{{\rm ap},C_2}^s(\rr d)$ for $C_2=\max(C,C_1,2^{2(s+1)})^2$. Using the fact that the inclusion $G_{{\rm ap},C_2}^s(\rr d) \subseteq G_{\rm ap}^s(\rr d)$ is continuous, we may conclude that the linear map $a(x,D): G_{{\rm ap},C_1}^s(\rr d) \mapsto G_{\rm ap}^s(\rr d)$ is continuous for any $C_1>0$. Finally, it follows from Lemma \ref{inductivecont1} that $a(x,D): G_{\rm ap}^s(\rr d) \mapsto G_{\rm ap}^s(\rr d)$ is continuous.
\end{proof}

\begin{cor}\label{gevreycont1}
If $a \in APS_{\rho,0}^{m,s}(\rr {2d})$ then $a(x,D)$ is continuous on $G_{\rm ap}^s(\rr d)$.
\end{cor}

Proposition \ref{gevreycont2} admits the extension of the domain of pseudodifferential operators $a(x,D)$ by means of duality for $a \in APS_{\rho,0}^{m,s}(\rr {3d})$, as we now explain.
If $a \in APS_{\rho,0}^{m,s}(\rr {3d})$ and $f \in G_{\rm ap}^s (\rr d)$, the function $a(x,D) f \in G_{\rm ap}^s (\rr d)$ is determined by
\begin{equation}\nonumber
\left\{ (a(x,D) f,g)_B = \sum_{\xi \in \rr d} (\wh{a(x,D)f})_\xi \ \overline{\wh g_\xi}, \quad g \in C_{\rm ap} (\rr d) \right\}.
\end{equation}
In fact, it suffices to take $g(x)=e^{2 \pi i x \cdot \xi}$ for all $\xi \in \rr d$, which gives the Fourier coefficients for $a(x,D) f$.
Moreover, \cite[Corollary 4.1]{Shubin4} shows that
$$
(a(x,D) f,g)_B = (f, b(x,D) g)_B, \quad f,g \in G_{\rm ap}^s (\rr d),
$$
where $b(x,D)$ is the formal adjoint of $a(x,D)$, defined by the amplitude $b(x,y,\xi) = \overline{a(y,x,\xi)}$.
In fact, the proof of \cite[Corollary 4.1]{Shubin4} treats symbols $a \in APS_{1,0}^m(\rr {2d})$, and it extends to amplitudes in $APS_{\rho,\delta}^{m,s}(\rr {3d})$.

Hence it is natural to define for $f \in (G_{\rm ap}^s)'(\rr d)$
\begin{equation}\label{duality1}
(a(x,D) f,g) = (f, b(x,D) g), \quad g \in G_{\rm ap}^s (\rr d),
\end{equation}
where $(\cdot,\cdot)$ denotes the duality between $(G_{\rm ap}^s)'$ and $G_{\rm ap}^s$. Proposition \ref{gevreycont2} shows that \eqref{duality1} is well defined, and $a(x,D): (G_{\rm ap}^s)' \mapsto (G_{\rm ap}^s)'$ is continuous with respect to the weak$^*$ topology.
Proposition \ref{dualembedding1} implies that
\begin{equation}\label{apriori1}
a(x,D): W_{s,0}^1(\rr d) \mapsto (G_{\rm ap}^s)'(\rr d)
\end{equation}
is a continuous map.

\section{Pseudodifferential calculus}\label{sectioncalculus}

In this section we develop a basic calculus for symbols in the classes $APS_{\rho,\delta}^{m,s}(\rr {2d})$ and operators acting on $G_{\rm ap}^s(\rr d)$. Throughout the section we assume that $s \geq 1$. The following results are adaptations of results in \cite[Chapter 3]{Rodino1}. As a general principle, we formulate almost periodic and global versions of the results in \cite[Chapter 3]{Rodino1}. The latter are local in the sense that estimates hold when the space variable $x$ belongs to a compact set, whereas here we keep track of uniform bounds for $x \in \rr d$. To benefit the reader we give full details for the most part of the proofs.

\begin{defn}\label{asymptoticexpansion1}
A \textit{formal sum} $\sum_{j \geq 0} a_j$ consists of a sequence $(a_j) \subseteq APS_{\rho,\delta}^{m,s}(\rr {2d})$ such that
\begin{equation}\nonumber
\begin{aligned}
|\pdd x \alpha \pdd \xi \beta a_j(x,\xi)|
\leq & \ C^{1+ |\alpha| + |\beta| + j} (j! \alpha!)^{s(\rho-\delta)} \beta! \eabs \xi^{m-\rho |\beta|+\delta |\alpha|-(\rho-\delta)j}, \\
& \alpha,\beta \in \nn d, \quad x \in \rr d, \quad \eabs \xi \geq B (j+|\beta|)^s, \quad j =0,1,\dots,
\end{aligned}
\end{equation}
for some $C,B>0$.
\end{defn}

A symbol $a \in APS_{\rho,\delta}^{m,s}(\rr {2d})$ is identified with the formal sum $\sum_{j \geq 0} a_j$ where $a_0=a$ and $a_j=0$ for $j>0$. Next we define an equivalence relation for formal sums.

\begin{defn}\label{equivalence1}
Given two formal sums $\sum_{j \geq 0} a_j$ and $\sum_{j \geq 0} b_j$,
then we write $\sum_{j \geq 0} a_j \sim \sum_{j \geq 0} b_j$ provided
\begin{equation}\nonumber
\begin{aligned}
\left| \pdd x \alpha \pdd \xi \beta \sum_{j<N} \left( a_j(x,\xi)-b_j(x,\xi) \right) \right|
\leq & \ C^{1+|\alpha|+|\beta|+N} (N! \alpha!)^{s(\rho-\delta)} \beta! \eabs \xi^{m-\rho |\beta|+\delta |\alpha|-(\rho-\delta)N}, \\
& \alpha,\beta \in \nn d, \quad \eabs \xi \geq B (N+|\beta|)^s, \quad N =1,2,\dots,
\end{aligned}
\end{equation}
for some $C,B>0$.
\end{defn}

\begin{lem}\label{formalsumconstruction1}
If $\sum_{j \geq 0} a_j$ is a formal sum, then there exists a symbol $a \in APS_{\rho,\delta}^{m,s}(\rr {2d})$ such that $a \sim \sum_{j \geq 0} a_j$.
\end{lem}
\begin{proof}
The proof is similar to \cite[Theorem 3.2.14]{Rodino1} so we only give a sketch.
The symbol $a$ is constructed as
\begin{equation}\nonumber
a(x,\xi) = \sum_{j=0}^\infty \varphi_j(\xi) a_j(x,\xi)
\end{equation}
where $(\varphi_j) \subseteq C_c^\infty(\rr d)$ is a sequence of cut-off functions that satisfy $0 \leq \varphi_j \leq 1$ and, for any given $R>0$ and integer $K>0$,
\begin{equation}\nonumber
\begin{aligned}
& \varphi_j(\xi)=0 \quad \mbox{if} \quad |\xi|<2 R (j+1)^s, \quad \varphi_j(\xi)=1 \quad \mbox{if} \quad |\xi|>3 R (j+1)^s, \\
& |\pd \gamma \varphi_j(\xi)| \leq C_K^{1+|\gamma|} (R (j+1)^{s-1})^{-|\gamma|} \quad \mbox{if} \quad |\gamma| \leq K (j+1),
\end{aligned}
\end{equation}
where $C_K>0$ depends on $K$ but not on $j,R,\gamma$. The proof of \cite[Theorem 3.2.14]{Rodino1} then shows that the result holds provided $K$ and $R$ are chosen sufficiently large.
\end{proof}

\begin{defn}
An operator $a(x,D)$ is said to be regularizing provided
\begin{equation}\nonumber
a(x,D): W_{s,0}^1(\rr d) \mapsto G_{\rm ap}^s(\rr d)
\end{equation}
continuously.
\end{defn}

\begin{prop}\label{zeroregularizing1}
If $a \in APS_{\rho,\delta}^{m,s}(\rr {2d})$ and $a \sim 0$ then $a(x,D)$ is regularizing.
\end{prop}
\begin{proof}
According to Definition \ref{equivalence1} we have
\begin{equation}\label{symbolestimate1}
\begin{aligned}
\left|\pdd x \beta  a(x,\xi) \right|
\leq & \ C^{1 + |\beta| + N} (N! \beta!)^{s(\rho-\delta)} \eabs \xi^{m  + \delta |\beta| - (\rho-\delta)N}, \\
& \beta \in \nn d, \quad \eabs \xi \geq B N^s, \quad N =1,2,\dots.
\end{aligned}
\end{equation}
According to \cite[Lemma 3.2.4]{Rodino1} we have
\begin{equation}\label{infimumestimate1}
\inf_{0 \leq N \leq B^{-1/s} \eabs{\xi}^{1/s}} C^N (N!)^{s(\rho-\delta)} \eabs{\xi}^{-(\rho-\delta)N }
\leq C_1 \exp\left( -3 \ep |\xi|^{1/s} \right)
\end{equation}
for some $C_1, \ep>0$.
Moreover, for $t \geq 0$ and $n \in \no$ we have for any $\theta>0$
$$
t^n = \left(\frac{\theta}{s} \right)^{- s n} \left( \frac{\theta}{s} \ t^{1/s} \right)^{n s}
\leq (\theta/s)^{- s n} (n!)^{s} \exp \left( \theta t^{1/s} \right),
$$
and hence
\begin{align}
& (\beta!)^{-s \delta} \eabs \xi^{\delta |\beta|} \leq C_{\theta}^{1+|\beta|} \exp\left( \theta |\xi|^{1/s} \right), \label{factorialestimate1} \\
& ((\alpha-\beta)!)^{-s} |\xi|^{|\alpha-\beta|} \leq C_{\theta}^{|\alpha - \beta|} \exp\left( \theta |\xi|^{1/s} \right), \quad \xi \in \rr d, \label{factorialestimate1b}
\end{align}
for some $C_\theta>0$. Inserting \eqref{infimumestimate1}, \eqref{factorialestimate1} and $\eabs \xi^{m} \leq D_\theta \exp\left( \theta |\xi|^{1/s} \right)$, where $D_\theta>0$, into \eqref{symbolestimate1} we obtain
\begin{equation}\label{symbolestimate2}
\left|\pdd x \beta  a(x,\xi) \right|
\leq C^{1+|\beta|} (\beta!)^ s \exp\left( - 2 \ep |\xi|^{1/s} \right)
\end{equation}
for some $C>0$, provided $\theta \leq \ep/2$.

Let $f \in TP(\rr d)$, i.e. $f(x) = \sum_{\xi \in \rr d} e^{2 \pi i \xi \cdot x} \wh f_\xi$ where the sum is finite.
Then we have (see e.g. \cite[Lemma 4]{Wahlberg1})
\begin{equation}\nonumber
a(x,D) f(x) = \sum_{\xi \in \rr d} e^{2 \pi i \xi \cdot x} a(x,\xi) \wh f_\xi \in C_{\rm ap}^\infty(\rr d),
\end{equation}
which gives, using \eqref{factorialestimate1b} and \eqref{symbolestimate2},
\begin{equation}\nonumber
\begin{aligned}
& \left| \pdd x \alpha a(x,D) f(x) \right|
= \left| \sum_{\xi \in \rr d} \sum_{\beta \leq \alpha} \binom{\alpha}{\beta} (2 \pi i \xi)^{\alpha-\beta} e^{2 \pi i \xi \cdot x} \pdd x \beta a(x,\xi) \wh f_\xi \right| \\
& \leq \sum_{\xi \in \rr d} \sum_{\beta \leq \alpha} | \wh f_\xi | \frac{(\alpha!)^s}{(\beta!)^s ((\alpha-\beta)!)^s} (2\pi)^{|\alpha-\beta|} |\xi|^{|\alpha-\beta|} C^{1+|\beta|} (\beta!)^ s \exp\left( - 2 \ep |\xi|^{1/s} \right) \\
& \leq C_1^{1 + |\alpha|} (\alpha!)^s \sum_{\xi \in \rr d} | \wh f_\xi | \exp\left( -\ep |\xi|^{1/s} \right) \\
& = C_1^{1 + |\alpha|} (\alpha!)^s \| f \|_{W_{s,\ep}^1(\rr d)}, \quad x \in \rr d,
\end{aligned}
\end{equation}
for some $C_1>0$. Since $W_{s,\ep}^1$ and $G_{{\rm ap},C_1}^s$ are a Banach spaces, and $TP(\rr d)$ is dense in $W_{s,\ep}^1$, $a(x,D)$ extends by density to a continuous map
$$
a(x,D): \ W_{s,\ep}^1 (\rr d) \mapsto G_{{\rm ap},C_1}^s (\rr d).
$$
Finally, since the inclusions $W_{s,0}^1 \subseteq W_{s,\ep}^1$ and $G_{{\rm ap},C_1}^s \subseteq G_{\rm ap}^s$ are continuous, the result follows.
\end{proof}

The next result corresponds to \cite[Theorem 3.2.24]{Rodino1}.
It gives a reformulation of
an operator with amplitude in $APS_{\rho,\delta}^{m,s}(\rr {3d})$ as an operator with symbol in $APS_{\rho,\delta}^{m,s}(\rr {2d})$, constructed from a formal sum based on the original amplitude,
modulo a regularizing operator.
We need the transpose $^t a(x,D)$ of an operator with amplitude $a \in APS_{\rho,\delta}^{m,s}(\rr {3d})$, which is defined by $^t a(x,D)=c(x,D)$ where $c(x,y,\xi)=a(y,x,-\xi)$, and satisfies
$$
( a(x,D) f, \overline{g} )_B = ( f, \overline{ c(x,D) g} )_B, \quad f,g \in G_{\rm ap}^s(\rr d).
$$

\begin{prop}\label{threevariablesasymp1}
If $a \in APS_{\rho,\delta}^{m,s}(\rr {3d})$ then $a(x,D)=b(x,D)+R$ where $b \sim \sum_{j \geq 0} a_j \in APS_{\rho,\delta}^{m,s}(\rr {2d})$,
\begin{equation}\nonumber
a_j(x,\xi) = (2 \pi i)^{-j} \sum_{|\alpha|=j} (\alpha!)^{-1} \pdd y \alpha \pdd \xi \alpha a(x,y,\xi) \big|_{y=x},
\end{equation}
and $R$, $^t R$ are regularizing.
\end{prop}
\begin{proof}
The Schwartz kernel of a linear operator $L: \mathscr S(\rr d) \mapsto \mathscr S'(\rr d)$ is the tempered distribution $K \in \mathscr S'(\rr {2d})$ that satisfies (cf. \cite{Hormander1})
$$
\langle Lf,g \rangle = \langle K,g \otimes f \rangle, \quad f,g \in \mathscr S (\rr d).
$$
The kernel $K$ of $a(x,D)$ for $a \in APS_{\rho,\delta}^{m,s}(\rr {3d})$ is
\begin{equation}\nonumber
K(x,y) = \int_{\rr d} e^{2 \pi i \xi \cdot (x-y)} a(x,y,\xi) {\,} d\xi
\end{equation}
when $m<-d$ (so that the integral makes sense), and by the corresponding partial Fourier transformation taken in $\mathscr S'$ otherwise.
According to the following Lemma \ref{kernelgevreyestimate1}, the kernel $K$ corresponding to $R=a(x,D)-b(x,D)$ satisfies
\begin{equation}\label{kernelestimate1}
|(x-y)^\gamma \pdd x \alpha \pdd y \beta K(x,y)| \leq C_\gamma C^{1 + |\alpha| + |\beta|} (\alpha ! \beta !)^s,
\end{equation}
for some constants $C,C_\gamma>0$, for all $x,y \in \rr d$ and all $\alpha,\beta,\gamma \in \nn d$
such that $|\gamma| \leq d+1$.

Next we define $L(x,y) = K(x,x-y)$.
We have
\begin{equation}\label{coordinatederivative1}
\pdd x \alpha \pdd y \beta L(x,y)
= \sum_{\sigma \leq \alpha} \binom{\alpha}{\sigma} (-1)^{|\beta|} \partial_1^{\alpha-\sigma} \partial_2^{\beta+\sigma} K(x,x-y).
\end{equation}
The estimate \eqref{kernelestimate1} and
$(\beta+\sigma) ! \leq 2^{|\beta|+|\sigma|} \beta! \sigma !$
give
\begin{equation}\label{kernelestimate2}
|y^\gamma \pdd x \alpha \pdd y \beta L(x,y)| \leq C_\gamma C^{1 + |\alpha| + |\beta|} (\alpha ! \beta !)^s, \quad x,y \in \rr d, \quad \alpha,\beta,\gamma \in \nn d, \quad |\gamma| \leq d+1,
\end{equation}
for some constants $C,C_\gamma>0$.
Hence we have
\begin{equation}\nonumber
\left| \pdd x \alpha \pdd y \beta L(x,y) \right|
\leq C^{1 + |\alpha| + |\beta|} (\alpha ! \beta !)^s \eabs{y}^{-d-1}
\end{equation}
where $C>0$, which gives the following result for the partial Fourier transform of $L$ in the second variable:
\begin{equation}\nonumber
\begin{aligned}
\left|  \xi^{\beta}(\mathscr F_{2} \pdd x \alpha  L) (x,\xi) \right|
& = \left| (2 \pi i)^{-|\beta|} \mathscr F_{2,y \rightarrow \xi} \left( \pdd x \alpha \pdd y \beta L (x,y) \right) \right| \\
& = (2 \pi )^{-|\beta|} \left| \int_{\rr d} \pdd x \alpha \pdd y \beta  L(x,y) e^{- 2 \pi i y \cdot \xi} dy \right| \\
& \leq C^{1 + |\alpha| + |\beta|} (\alpha ! \beta !)^s, \quad \alpha,\beta \in \nn d, \quad C>0, \quad x,\xi \in \rr d.
\end{aligned}
\end{equation}
Since $K(x,y)=L(x,x-y)$ we may interchange the roles of $K$ and $L$ in \eqref{coordinatederivative1}, and we obtain
\begin{equation}\nonumber
\begin{aligned}
\left| \xi^{\beta} (\mathscr F_2 \pdd x \alpha  K)(x,\xi) \right|
& = \left| (2 \pi i )^{-|\beta|} (\mathscr F_2 \pdd x \alpha \pdd y \beta K)(x,\xi) \right| \\
& \leq (2 \pi)^{-|\beta|} \sum_{\sigma \leq \alpha} \binom{\alpha}{\sigma} \left| \mathscr F_2  \partial_1^{\alpha-\sigma} \partial_2^{\beta+\sigma} L  (x,-\xi) \right| \\
& = (2 \pi)^{-|\beta|} \sum_{\sigma \leq \alpha} \binom{\alpha}{\sigma} \left| (-2 \pi i \xi)^{\beta+\sigma} \mathscr F_2  \partial_1^{\alpha-\sigma} L  (x,-\xi) \right| \\
& \leq (2 \pi)^{|\alpha|} \sum_{\sigma \leq \alpha} \binom{\alpha}{\sigma}
C^{1+|\alpha-\sigma|+|\beta+\sigma|} ((\alpha-\sigma) ! (\beta+\sigma) !)^s \\
& \leq C_1^{1 + |\alpha| + |\beta|} (\alpha ! \beta !)^s, \quad x,\xi \in \rr d,
\end{aligned}
\end{equation}
where $C_1>0$.

According to the results in \cite[Chapter IV, Section 2.1]{Gelfand1}, this implies the estimate
\begin{equation}\label{partialfourierkernel1}
\left| \mathscr F_{2} \pdd x \alpha  K (x,\xi) \right|
\leq  C^{1 + |\alpha|} (\alpha !)^s \exp(- \ep |\xi|^{1/s})
\end{equation}
for some $C>0$ and some $\ep>0$.

Finally, let $f \in TP(\rr d)$ which means that $f$ is a finite sum of the form
$$
f(y) = \sum_{\xi \in \rr d} \wh f_\xi \ e^{2 \pi i \xi \cdot y}.
$$
Using \eqref{partialfourierkernel1} we obtain
\begin{equation}\nonumber
\begin{aligned}
\left| \pd \alpha Rf (x) \right| & = \left| \int_{\rr d} \pdd x \alpha K(x,y) f(y) dy \right|
= \left|  \sum_{\xi \in \rr d} \wh f_\xi \mathscr F_2 \pdd x \alpha K(x,-\xi) \right| \\
& \leq C^{1 + |\alpha|} (\alpha !)^s \sum_{\xi \in \rr d} \exp(- \ep |\xi|^{1/s}) |\wh f_\xi| \\
& = C^{1 + |\alpha|} (\alpha !)^s \| f \|_{W_{s,\ep}^1}, \quad x \in \rr d, \quad C>0.
\end{aligned}
\end{equation}
Since it is clear that $R f = (a(x,D)-b(x,D)) f$ belongs to $C_{\rm ap}(\rr d)$, this proves that $R$ is regularizing. Similarly, it follows from \eqref{kernelestimate1} that also $^t R$ is regularizing, since its kernel is $^t K(x,y)=K(y,x)$.
\end{proof}

It remains to state and prove the lemma upon which Proposition \ref{threevariablesasymp1} is based.

\begin{lem}\label{kernelgevreyestimate1}
Suppose $a \in APS_{\rho,\delta}^{m,s}(\rr {3d})$,
\begin{equation}\nonumber
a_j(x,\xi) = (2 \pi i)^{-j} \sum_{|\alpha|=j} (\alpha!)^{-1} \pdd y \alpha \pdd \xi \alpha a(x,y,\xi) \big|_{y=x},
\end{equation}
$b \sim \sum_{j \geq 0} a_j$.
Then the kernel $K$ of $a(x,D)-b(x,D)$ satisfies
\begin{equation}\label{kernelestimate3}
|(x-y)^\gamma \pdd x \alpha \pdd y \beta K(x,y)| \leq C_\gamma C^{1 + |\alpha| + |\beta|} (\alpha ! \beta !)^s,
\end{equation}
for some constants $C,C_\gamma>0$, for all $x,y \in \rr d$ and all $\alpha,\beta,\gamma \in \nn d$
such that $|\gamma| \leq d+1$.
\end{lem}
\begin{proof}
The proof is based on the ideas and techniques in the proofs of \cite[Theorems 3.2.6, 3.2.23 and 3.2.24]{Rodino1}.

Let $\ep>0$.
For $(x,y) \in \rr {2d}$ such that $|x-y| < \ep$, the proof of \cite[Theorem 3.2.24]{Rodino1} gives the estimate \eqref{kernelestimate3}.

For $(x,y) \in \rr {2d}$ such that $|x-y| \geq \ep$,
the estimate \eqref{kernelestimate3} follows using arguments as in the proof of \cite[Theorem 3.2.6]{Rodino1}, somewhat modified, as we now detail. The estimate holds for any kernel $K$ of an operator $a(x,D)$ where $a \in APS_{\rho,\delta}^{m,s}(\rr {3d})$.
Given $R>0$ we define a sequence of cut-off functions $\{ \psi_N \}_{N \geq 0} \subseteq C_c^\infty(\rr d)$ that partition unity, $\sum_{N \geq 0} \psi_N \equiv 1$, and satisfy for some $C>0$
\begin{equation}\label{partitionofunity1}
\begin{aligned}
& \mbox{(i)} \quad \supp \psi_0 \subseteq \{ \xi: |\xi| \leq 3 R \}, \\
& \mbox{(ii)} \quad \supp \psi_N \subseteq \{ 2 R N^s \leq |\xi| \leq 3 R (N+1)^s \}, \quad N>0, \\
& \mbox{(iii)} \quad |\pd \beta \psi_N(\xi)| \leq C^{1+|\beta|} (R N^{s-1})^{-|\beta|} \quad \mbox{for} \quad |\beta| \leq 2 N, \quad N \geq 0.
\end{aligned}
\end{equation}
See \cite[Lemma 3.2.7]{Rodino1}.
Then we set
\begin{equation}\nonumber
K_N(x,y) = \int_{\rr d} e^{2 \pi i \xi \cdot (x-y)} \psi_N(\xi) a(x,y,\xi) {\,} d\xi
\end{equation}
which implies that $K = \sum_{N \geq 0} K_N$. We have
\begin{equation}\nonumber
\begin{aligned}
& \pdd x \alpha \pdd y \beta K_N(x,y) \\
& = \sum_{\la \leq \beta, \ \mu \leq \alpha}
\binom{\alpha}{\mu} \binom{\beta}{\la}
 \int_{\rr d} e^{2 \pi i \xi \cdot (x-y)}
(-1)^{|\beta-\la|}(2 \pi i \xi)^{\alpha-\mu+\beta-\la}
\psi_N(\xi) \pdd x \mu \pdd y \la a(x,y,\xi) {\,} d\xi.
\end{aligned}
\end{equation}
Since $|x-y|\geq\ep>0$ there exists $j$, $1 \leq j \leq d$ such that $|x_j-y_j| \geq C_\ep > 0$.
Let $e_j \in \nn d$ denote the $j$th basis vector in $\rr d$ and set $\gamma_N=\gamma+N e_j$.
Integration by parts gives
\begin{equation}\nonumber
\begin{aligned}
& \left| (x-y)^\gamma \pdd x \alpha \pdd y \beta K_N(x,y) \right|
= \left| (x-y)^{- N e_j + \gamma_N} \pdd x \alpha \pdd y \beta K_N(x,y) \right| \\
= & \ \Big| (x_j-y_j)^{-N} \sum_{\la \leq \beta, \ \mu \leq \alpha}
\binom{\alpha}{\mu} \binom{\beta}{\la}
 \int_{\rr d} e^{2 \pi i \xi \cdot (x-y)}
(-1)^{|\beta-\la|} (-2 \pi i)^{-|\gamma|-N}  \\
& \times \partial_\xi^{\gamma_N} \left[ (2 \pi i \xi)^{\alpha-\mu+\beta-\la}
\psi_N(\xi) \pdd x \mu \pdd y \la a(x,y,\xi) \right] {\,} d\xi \Big| \\
\leq & \ C_\gamma C^{1+|\alpha|+|\beta|+N}
\sum_{\stackrel{\la \leq \beta, \ \mu \leq \alpha}{\kappa \leq \gamma_N, \ \nu \leq \kappa}}
\binom{\alpha}{\mu} \binom{\beta}{\la} \binom{\gamma_N}{\kappa} \binom{\kappa}{\nu} \\
& \times \int_{\rr d} \left|
\pd \nu \xi^{\alpha-\mu+\beta-\la} \
\partial^{\kappa-\nu} \psi_N(\xi) \
\pdd x \mu \pdd y \la \partial_\xi^{\gamma_N-\kappa} a(x,y,\xi) \right| {\,} d\xi,
\end{aligned}
\end{equation}
for some $C,C_\gamma>0$.

First we assume $N>0$ and $N \geq |\gamma|$.

Then $|\kappa-\nu| \leq |\gamma_N| \leq 2N$ and we may use \eqref{partitionofunity1} (iii) to estimate $\partial^{\kappa-\nu} \psi_N(\xi)$. Moreover, for $\xi \in \supp \psi_N$ we have $\eabs{\xi} \geq 2 R N^s = R 2^{1-s}(2N)^s \geq R 2^{1-s} |\gamma_N-\kappa|^s \geq B |\gamma_N-\kappa|^s$ provided $R \geq B 2^{s-1}$, so we may use \eqref{symbclass3a}. We may assume $\nu \leq \alpha-\mu+\beta-\la$. Thus we obtain
\begin{equation}\label{integralestimate1}
\begin{aligned}
& \int_{\rr d} \left| \pd \nu \xi^{\alpha-\mu+\beta-\la} \
\partial^{\kappa-\nu} \psi_N(\xi) \ \pdd x \mu \pdd y \la \partial_\xi^{\gamma_N-\kappa} a(x,y,\xi) \right| {\,} d\xi \\
\leq & \ \frac{(\alpha-\mu+\beta-\la)!}{(\alpha-\mu+\beta-\la- \nu)!}
(3 R (N+1)^s)^{|\alpha-\mu|+|\beta-\la|-|\nu|} \\
& \times C^{1+|\kappa-\nu|+|\mu|+|\la|+|\gamma_N-\kappa|} (R N^{s-1})^{-|\kappa-\nu|} \\
& \times
(\mu! \la!)^{s(\rho-\delta)} (\gamma_N-\kappa)! (4 R (N+1)^s)^{m-\rho |\gamma_N-\kappa|+\delta(|\mu|+|\la|)} \int_{\supp \psi_N} d \xi,
\end{aligned}
\end{equation}
assuming $m - \rho |\gamma_N-\kappa|+\delta(|\mu|+|\la|) \geq 0$ and $R \geq 1$.
(If instead $m-\rho |\gamma_N-\kappa|+\delta(|\mu|+|\la|) < 0$ then $4 R (N+1)^s$ above is replaced by $2 R N^s$, and the following argument works also in this case.)
Using
\begin{equation}\nonumber
\int_{\supp \psi_N} d \xi \leq C (3 R (N+1)^s)^d,
\end{equation}
we may bound the factor of \eqref{integralestimate1} that depends on $R$ as
\begin{equation}\nonumber
\begin{aligned}
R^{|\alpha|+|\beta|-|\nu|-|\kappa-\nu|+m-\rho |\gamma_N-\kappa|+\delta (|\mu|+|\la|)+d}
& \leq R^{d+m+(1+\delta)(|\alpha|+|\beta|)-\rho |\gamma_N|+|\kappa|(\rho-1)} \\
& \leq R^{d+m+(1+\delta)(|\alpha|+|\beta|)-\rho |\gamma|} R^{-\rho N}.
\end{aligned}
\end{equation}
Using
\begin{equation}\nonumber
(\alpha!)^{-1} (N+1)^{|\alpha|} \leq e^{d(N+1)}, \quad (\alpha + \beta)! \leq  2^{|\alpha|+|\beta|} \alpha! \beta!,
\end{equation}
and $s \rho \geq s(\rho-\delta) \geq 1$,
we hence obtain for some $C,C_R,C_\gamma>0$ (varying over the inequalities), where only $C_R$ depend on $R$,
\begin{equation}\label{estimate1}
\begin{aligned}
& \left| (x-y)^\gamma \pdd x \alpha \pdd y \beta K_N(x,y) \right| \\
\leq & \ C_\gamma C_R^{1+|\alpha|+|\beta|} (\alpha! \beta!)^s (C R^{- \rho})^N
\sum_{\stackrel{\la \leq \beta, \ \mu \leq \alpha}{\kappa \leq \gamma_N, \ \nu \leq \kappa}}
((\alpha-\mu)! (\beta-\la)!)^{-s}
\frac{\gamma_N!}{(\kappa-\nu)!} \\
& \times
(\mu! \la!)^{-s \delta}
((N+1)^s)^{|\alpha-\mu|+|\beta-\la|-|\nu|+m-\rho |\gamma_N-\kappa|+\delta(|\mu|+|\la|)}
N^{(1-s)|\kappa-\nu|} \\
\leq & \ C_\gamma C_R^{1+|\alpha|+|\beta|} (\alpha! \beta!)^s (C R^{- \rho})^N
\sum_{\stackrel{\la \leq \beta, \ \mu \leq \alpha}{\kappa \leq \gamma_N, \ \nu \leq \kappa}}
\gamma_N! ((N+1)^s)^{-|\nu|-\rho |\gamma_N-\kappa|} N^{-s|\kappa-\nu|} \\
\leq & \ C_\gamma C_R^{1+|\alpha|+|\beta|} (\alpha! \beta!)^s (C R^{- \rho})^N
\sum_{\stackrel{\la \leq \beta, \ \mu \leq \alpha}{\kappa \leq \gamma_N, \ \nu \leq \kappa}}
\gamma! N! N^{-s (\rho |\gamma|+\rho N +|\kappa|(1-\rho))} \\
\leq & \ C_\gamma C_R^{1+|\alpha|+|\beta|} (\alpha! \beta!)^s (C R^{- \rho})^N
\sum_{\stackrel{\la \leq \beta, \ \mu \leq \alpha}{\kappa \leq \gamma_N, \ \nu \leq \kappa}}
N^{N (1-\rho s) - s \rho |\gamma|} \\
\leq & \ C_\gamma  C_R^{1+|\alpha|+|\beta|} (\alpha! \beta!)^s (C R^{- \rho})^N
\sum_{\stackrel{\la \leq \beta, \ \mu \leq \alpha}{\kappa \leq \gamma_N, \ \nu \leq \kappa}}
\binom{\beta}{\la} \binom{\alpha}{\mu} \binom{\gamma_N}{\kappa} \binom{\kappa}{\nu} \\
\leq & \ C_\gamma  C_R^{1+|\alpha|+|\beta|} (\alpha! \beta!)^s (C R^{- \rho})^N.
\end{aligned}
\end{equation}
Since $\rho>0$ we may choose $R$ such that $C R^{- \rho} \leq 1/2$.
Thus we have a useful estimate for $N>0$ and $N \geq |\gamma|$.

If $N<|\gamma|$ we may, since $|\gamma-\kappa|\leq |\gamma| \leq d+1 \leq 2M$,
combine \eqref{symbclass3a} and \eqref{symbclass3ap} into the estimate, $\xi \in \supp \psi_N$,
\begin{equation}\nonumber
|\pdd x \mu \pdd y \la \partial_\xi^{\gamma-\kappa}   a(x,y,\xi)|
\leq C^{1+ |\mu| + |\la| + |\gamma|} (\mu! \la!)^{s(\rho-\delta)} (\gamma-\kappa)! (R(N+1)^s)^{m + \delta |\mu+\la|}.
\end{equation}
Moreover, we have $|\partial^{\kappa-\nu} \psi_N(\xi)| \leq C_\gamma$ for $\nu\leq \kappa$, $\kappa \leq \gamma$ and $N<|\gamma|$.
Arguing as above (without splitting $\gamma$ into $\gamma=\gamma_N- N e_j$) this gives
\begin{equation}\label{estimate2}
\left| (x-y)^\gamma \pdd x \alpha \pdd y \beta K_N(x,y) \right|
\leq C_\gamma  C_R^{1+|\alpha|+|\beta|} (\alpha! \beta!)^s,
\end{equation}
and a similar estimate can be shown to hold also when $N=|\gamma|=0$.
Combining \eqref{estimate1} for $N \geq |\gamma|$ and $C R^{- \rho} \leq 1/2$, and \eqref{estimate2} for $N<|\gamma|$
proves \eqref{kernelestimate3} when $|x-y|\geq\ep>0$, in view of $K = \sum_{N \geq 0} K_N$.
\end{proof}

The next topic is the calculus of composition of operators. First we need some preparations before we prove the main result Proposition \ref{compositioncalculus1}.

\begin{defn}
Let $\sum_{j \geq 0} a_j \in APS_{\rho,\delta}^{m,s}(\rr {2d})$, $\sum_{j \geq 0} b_j \in APS_{\rho,\delta}^{n,s}(\rr {2d})$ be formal sums. The \textit{symbol product} $(\sum_{j \geq 0} a_j) \circ (\sum_{j \geq 0} b_j )$ is defined as the formal sum $\sum_{j \geq 0} c_j$, where
\begin{equation}\label{symbolproduct1}
c_j(x,\xi) = \sum_{|\alpha|+k+l=j} (\alpha!)^{-1} (2 \pi i)^{-|\alpha|}  \ \pdd \xi \alpha a_k (x,\xi) \ \pdd x \alpha b_l (x,\xi).
\end{equation}
\end{defn}

It can be verified that $(c_j) \subseteq APS_{\rho,\delta}^{m+n,s}(\rr {2d})$ and the formal sum $\sum_{j \geq 0} c_j$, where $c_j$ is defined by \eqref{symbolproduct1}, is well defined according to Definition \ref{asymptoticexpansion1}. By Lemma \ref{formalsumconstruction1} there exists $c \sim a \circ b \in APS_{\rho,\delta}^{m+n,s}(\rr {2d})$.

\begin{lem}\label{compositionlemma1}
Let $\chi \in C_c^\infty(\rr d)$ equal one in a neighborhood of the origin and $\chi_\delta(x)=\chi(\delta x)$,
and let $f \in C(\rr d)$ satisfy the bound $|f(\xi)| \leq C \eabs{\xi}^{N}$, $\xi \in \rr d$, for some $C,N>0$.
Then
\begin{equation}\nonumber
\lim_{\delta \longrightarrow +0}
\int_{\rr d} \wh{\chi_\delta} (\xi-\eta) \chi_\delta (\eta) f(\eta) d\eta = f(\xi), \quad \xi \in \rr d.
\end{equation}
\end{lem}
\begin{proof}
Since $\int_{\rr d} \wh{\chi_\de}(\xi) d\xi=1$ we have
\begin{equation}\nonumber
\begin{aligned}
& \left| \int_{\rr d} \wh{\chi_\delta} (\xi-\eta) \chi_\delta (\eta) f(\eta) d\eta - f(\xi) \right| \\
& = \left| \int_{\rr d} \wh{\chi_\delta} (\eta) \Big( f(\xi-\eta) - f(\xi) + f(\xi-\eta) (\chi_\delta (\xi-\eta) - 1) \Big) d\eta \right| \\
& \leq \int_{\rr d} | \wh{\chi_\delta} (\eta)| \ |f(\xi-\eta) - f(\xi)| d\eta
+ \int_{\rr d} | \wh{\chi_\delta} (\eta)| \ |f(\xi-\eta)| \ |(\chi_\delta (\xi-\eta) - 1)| d\eta \\
& := I_1 + I_2.
\end{aligned}
\end{equation}
For $\theta>0$ we split $I_1$ as
\begin{equation}\nonumber
I_1 =
\int_{|\eta| \leq \theta} | \wh{\chi_\delta} (\eta)| \ |f(\xi-\eta) - f(\xi)| d\eta
+ \int_{|\eta| > \theta} | \wh{\chi_\delta} (\eta)| \ |f(\xi-\eta) - f(\xi)| d\eta.
\end{equation}
Let $\ep>0$. The first integral is not larger than $\ep$, for $\theta$ sufficiently small, by the assumed continuity of $f$ and $\| \wh{\chi_\delta} \|_{L^1} = \| \wh{\chi} \|_{L^1}$.
Using $\widehat{\chi_\delta}(\eta) = \delta^{-d} \widehat{\chi}(\eta/\delta)$ and assuming $\delta \leq 1$, the second integral can be written
\begin{equation}\nonumber
\begin{aligned}
\int_{|\eta| > \theta/\delta} | \wh{\chi} (\eta)| \ |f(\xi-\delta \eta) - f(\xi)| d\eta
& \leq C \int_{|\eta| > \theta/\delta} | \wh{\chi} (\eta)| \eabs{\xi}^N (\eabs{\delta \eta}^N + 1) d\eta \\
& \leq 2 \ C \eabs{\xi}^N \int_{|\eta| > \theta/\delta} \eabs{\eta}^N | \wh{\chi} (\eta)| d\eta
\end{aligned}
\end{equation}
which approaches zero as $\delta \rightarrow +0$, because $\widehat \chi \in \mathscr S$.
This shows that $I_1$ approaches zero as $\delta \rightarrow +0$.

It remains to estimate $I_2$.
We have $\chi_\delta(\xi)=1$ for $|\xi| < \theta/\delta$, for some $\theta>0$.
Let $B>0$ be arbitrary.
Then $\{ \eta: \  |\xi-\eta| \geq \theta/\delta \} \subseteq \{ \eta: \ |\eta| \geq B \}$ if $\delta>0$ is sufficiently small.
Since $\wh \chi \in \mathscr S$ we have
$|\widehat{\chi_\delta}(\eta)| = |\delta^{-d} \widehat{\chi}(\eta/\delta)| \leq C_n \delta^{n-d} |\eta|^{-n}$
for any $n>N+d$.
This gives
\begin{equation}\nonumber
\begin{aligned}
I_2 & = \int_{|\xi-\eta| \geq \theta/\delta} | \wh{\chi_\delta} (\eta)| \ |f(\xi-\eta)| \ |(\chi_\delta (\xi-\eta) - 1)| d\eta \\
& \leq C \int_{|\eta| \geq B} \delta^{n-d} |\eta|^{-n}  \eabs{\xi}^N \eabs{\eta}^N  d\eta
\leq C \eabs{\xi}^N \delta^{n-d} \int_{\rr d} \eabs{\eta}^{N-n}  d\eta.
\end{aligned}
\end{equation}
Hence $I_2$ approaches zero as $\de \rightarrow +0$.
\end{proof}

At this point we are in a position to state and prove the composition theorem. The corresponding result for operators acting on the space of compactly supported Gevrey functions $G_c^s(\rr d)$ ($s>1$) is \cite[Theorem 3.2.20]{Rodino1}.
Here we assume $\delta=0$.

\begin{prop}\label{compositioncalculus1}
Let $a \in APS_{\rho,0}^{m,s}(\rr {2d})$ and $b \in APS_{\rho,0}^{n,s}(\rr {2d})$. Then $a(x,D) b(x,D) = c(x,D) + R$ where $R$ is regularizing and $c \sim a \circ b$.
\end{prop}
\begin{proof}
For $b \in APS_{\rho,0}^{n,s}(\rr {2d})$ we have $^t b(x,D) = b_1(x,D)$ where
$$
b_1(x,y,\xi)=b(y,-\xi) \in APS_{\rho,0}^{n,s}(\rr {3d}).
$$
By Proposition \ref{threevariablesasymp1} we have $^t b(x,D) = b'(x,D) + R_1$ where $R_1, {^t} R_1$ are regularizing, and $b' \in APS_{\rho,0}^{n,s}(\rr {2d})$, $b' \sim \sum_{j \geq 0} b_j$,
\begin{equation}\label{bprimeasymp1}
b_j(x,\xi) = (-2 \pi i)^{-j} \sum_{|\alpha|=j} (\alpha!)^{-1} \pdd \xi \alpha \pdd x \alpha b(x,-\xi), \quad j =0,1 \dots.
\end{equation}
Now we have
\begin{equation}\label{operatorcomposition1}
\begin{aligned}
a(x,D) b(x,D) & = a(x,D) \ {^t} ({^t} b(x,D)) = a(x,D) \ {^t} b'(x,D) +  a(x,D) \ {^t} R_1 \\
& = a(x,D) \ {^t} b'(x,D) + R_2
\end{aligned}
\end{equation}
where $R_2$ is regularizing, since $a(x,D)$ is continuous on $G_{\rm ap}^s(\rr d)$ according to Corollary \ref{gevreycont1}.
The operator ${^t} b'(x,D)$ has amplitude $b'(y,-\xi) \in APS_{\rho,0}^{n,s}(\rr {3d})$.
A combination of \eqref{psdopartint1} and \eqref{psdopartint3} thus gives for $f \in G_{\rm ap}^s(\rr d)$,
$2M>d$ and $2N>d+m$,
\begin{equation}\label{compositionargument1}
\begin{aligned}
& a(x,D) \ {^t} b'(x,D) f(x) = \iint_{\rr {2d}} e^{2 \pi i \xi \cdot (x-y)}
\eabs{\xi}^{-2N} (1-\Delta_\xi)^M a(x,\xi) \\
& \times (1-\Delta_y)^N( \eabs{x-y}^{-2M} \ {^t} b'(y,D) f(y) ) {\,} dy {\,} d\xi \\
= & \iiiint_{\rr {4d}} e^{2 \pi i \xi \cdot (x-y)}
\eabs{\xi}^{-2N} (1-\Delta_\xi)^M a(x,\xi) \\
& \times (1-\Delta_y)^N
\Big(
\eabs{x-y}^{-2M} e^{2 \pi i \eta \cdot (y-z)} \eabs{\eta}^{-2N} \\
& \times (1-\Delta_z)^{N} \left[ \eabs{y-z}^{-2M} f(z) (1-\Delta_\eta)^{M} b'(z,-\eta) \right] \Big) {\,} dz {\,} d\eta {\,} dy {\,} d\xi \\
= & \lim_{\ep \rightarrow +0} \lim_{\delta \rightarrow +0} \iiiint_{\rr {4d}}
\chi_\delta (y) \chi_\delta (\eta)  \chi_\ep (z) \chi_\ep (\xi)
\ e^{2 \pi i \left(\xi \cdot (x-y) + \eta \cdot (y-z)\right)} \\
& \times a(x,\xi) \ b'(z,-\eta) f(z) {\,} dy {\,} d\eta {\,} dz {\,} d\xi \\
= & \lim_{\ep \rightarrow +0} \lim_{\delta \rightarrow +0} \iint_{\rr {2d}} \left( \int_{\rr {d}}
 \widehat{\chi_\delta} (\xi-\eta) \chi_\delta (\eta) e^{-2 \pi i \eta \cdot z} \ b'(z,-\eta) {\,} d\eta \right) \\
 & \times \chi_\ep (z) \chi_\ep (\xi)
e^{2 \pi i \xi \cdot x} a(x,\xi) f(z) {\,} dz {\,} d\xi,
\end{aligned}
\end{equation}
where $\chi \in C_c^\infty(\rr d)$ equals one in a neighborhood of the origin and $\chi_\ep(x)=\chi(\ep x)$.
First we study the inner integral for $\ep>0$ fixed.
The inner integral can be split and estimated as
\begin{equation}\label{innerintegral1}
\begin{aligned}
& \left| \int_{\rr {d}} \widehat{\chi_\delta} (\xi-\eta) \chi_\delta (\eta) e^{-2 \pi i \eta \cdot z} b'(z,-\eta) {\,} d\eta \right| \\
& \leq C \int_{|\xi-\eta| \leq 1} |\widehat{\chi_\delta}(\xi-\eta)| \ |b'(z,-\eta)| d\eta
+ C \int_{|\xi-\eta| \geq 1} |\widehat{\chi_\delta}(\xi-\eta)| \ |b'(z,-\eta)| d\eta.
\end{aligned}
\end{equation}
Since $\xi$ in the outer integral of \eqref{compositionargument1} belongs to a compact set, we have $|\eta| \leq C$ for some $C>0$ in the first integral of \eqref{innerintegral1}.
Using $\widehat{\chi_\delta}(\xi) = \delta^{-d} \widehat{\chi}(\xi/\delta)$, the first integral can hence be estimated by a constant times $\| \widehat \chi \|_{L^1}$.

Since $\widehat{\chi} \in \mathscr S(\rr d)$ and $b' \in APS_{\rho,0}^{n,s}(\rr {3d})$, the second integral of \eqref{innerintegral1} can be estimated for any $L>0$ as
\begin{equation}\nonumber
\begin{aligned}
& C \int_{|\xi-\eta| \geq 1} |\widehat{\chi_\delta}(\xi-\eta)| \ |b'(z,-\eta)| {\,} d\eta
\leq C \int_{|\xi-\eta| \geq 1} \delta^{L-d} |\xi-\eta|^{-L} \eabs{\eta}^n {\,} d\eta \\
& \leq C \int_{\rr d} \delta^{L-d} \eabs{\xi-\eta}^{-L} \eabs{\eta}^n {\,} d\eta
\leq C \delta^{L-d} \eabs{\xi}^L \int_{\rr d} \eabs{\eta}^{-L} \eabs{\eta}^n {\,} d\eta
\leq C \eabs{\xi}^L
\end{aligned}
\end{equation}
provided $L > n+d$ and $\delta \leq 1$.

We have now proved that the inner integral \eqref{innerintegral1} is bounded by a constant for $\delta \leq 1$, and therefore Lebesgue's dominated convergence theorem admits us to rewrite \eqref{compositionargument1} as
\begin{equation}\nonumber
\begin{aligned}
& a(x,D) \ {^t} b'(x,D) f(x)
= \lim_{\ep \rightarrow +0}  \iint_{\rr {2d}} \lim_{\delta \rightarrow +0} \left( \int_{\rr {d}}
 \widehat{\chi_\delta} (\xi-\eta) \chi_\delta (\eta) e^{-2 \pi i \eta \cdot z} b'(z,-\eta) {\,} d\eta \right) \\
 & \times \chi_\ep (z) \chi_\ep (\xi)
\ e^{2 \pi i \xi \cdot x} a(x,\xi) f(z) {\,} dz {\,} d\xi \\
= & \lim_{\ep \rightarrow +0}  \iint_{\rr {2d}}
\chi_\ep (z) \chi_\ep (\xi)
e^{2 \pi i \xi \cdot (x-z)} a(x,\xi) \ b'(z,-\xi) f(z) {\,} dz {\,} d\xi,
\end{aligned}
\end{equation}
in the second step using Lemma \ref{compositionlemma1}.
This means that $a(x,D) \ {^t} b'(x,D) = c'(x,D)$ where
$$
c'(x,y,\xi) = a(x,\xi) \  b'(y,-\xi) \in APS_{\rho,0}^{m+n,s}(\rr {3d}).
$$
Hence Proposition \ref{threevariablesasymp1} gives
\begin{equation}\label{operatorcomposition2}
a(x,D) \ {^t} b'(x,D) = c'(x,D)=c(x,D) + R_3
\end{equation}
where $R_3$, $^t R_3$ are regularizing, $c \in APS_{\rho,0}^{m+n,s}(\rr {2d})$, $c \sim \sum_{j=0}^\infty c_j$ and
\begin{equation}\label{cj1}
\begin{aligned}
c_j(x,\xi)
& = (2 \pi i)^{-j} \sum_{|\gamma|=j} (\gamma!)^{-1}  \pdd \xi \gamma \pdd y \gamma \left[ a(x,\xi) b'(y,-\xi) \right] \Big|_{y=x} \\
& = (2 \pi i)^{-j} \sum_{|\gamma|=j} (\gamma!)^{-1} \pdd \xi \gamma \left[ a(x,\xi) \pdd x \gamma  b'(x,-\xi) \right].
\end{aligned}
\end{equation}
From $b' \sim \sum_{j \geq 0} b_j$ we have
\begin{equation}\label{bprime1}
\begin{aligned}
\partial_x^{\gamma+\mu} \pdd \xi \la & \left( b'(x,-\xi) -\sum_{k<N-j} b_k (x,-\xi) \right) :=  F_{\gamma+\mu,\la,N-j}(\xi), \\
\left| F_{\gamma+\mu,\la,N-j} (\xi) \right|
\leq & \ C^{1+|\gamma|+|\mu|+|\la|+N-j} ((N-j)! (\gamma+\mu)!)^{s \rho} \la! \eabs \xi^{n-\rho |\la|-\rho (N-j)}, \\
& \eabs \xi \geq B (N-j+|\la|)^s, \quad N=j+1,j+2,\dots.
\end{aligned}
\end{equation}
Differentiation of \eqref{cj1} and insertion of \eqref{bprime1} and \eqref{bprimeasymp1} give
\begin{equation}\label{cj2}
\begin{aligned}
& \pdd x \alpha \pdd \xi \beta c_j(x,\xi)
= (2 \pi i)^{-j} \sum_{|\gamma|=j} (\gamma!)^{-1}
\pdd x \alpha \partial_\xi^{\beta+\gamma} \left[ a(x,\xi) \pdd x \gamma  b'(x,-\xi) \right] \\
= & \ (2 \pi i)^{-j} \sum_{|\gamma|=j} (\gamma!)^{-1}
\sum_{\stackrel{\mu \leq \alpha}{\la \leq \beta+\gamma}}
\binom{\alpha}{\mu} \binom{\beta+\gamma}{\la}
\partial_x^{\alpha-\mu} \partial_\xi^{\beta+\gamma-\la} a(x,\xi)
\partial_x^{\gamma+\mu} \partial_\xi^{\la} \left( b'(x,-\xi) \right) \\
= & \ (2 \pi i)^{-j} \sum_{|\gamma|=j} (\gamma!)^{-1}
\sum_{\stackrel{\mu \leq \alpha}{\la \leq \beta+\gamma}}
\binom{\alpha}{\mu} \binom{\beta+\gamma}{\la}
\partial_x^{\alpha-\mu} \partial_\xi^{\beta+\gamma-\la} a(x,\xi)
F_{\gamma+\mu,\la,N-j} (\xi) \\
& + \sum_{|\gamma|=j} \ \ \sum_{k<N-j} \ \ \sum_{|\kappa|=k} (2 \pi i)^{-(j+k)}
(\gamma! \kappa!)^{-1}
\sum_{\stackrel{\mu \leq \alpha}{\la \leq \beta+\gamma}}
\binom{\alpha}{\mu} \binom{\beta+\gamma}{\la} \\
& \times
\partial_x^{\alpha-\mu} \partial_\xi^{\beta+\gamma-\la} a(x,\xi)
\partial_x^{\gamma+\mu+\kappa} \partial_\xi^{\la+\kappa} b(x,\xi) (-1)^{|\kappa|}.
\end{aligned}
\end{equation}
On the other hand, if we define
\begin{equation}\nonumber
\begin{aligned}
d_j(x,\xi)
& = (2 \pi i)^{-j} \sum_{|\gamma|+|\kappa|=j} (\gamma! \kappa!)^{-1}
\pdd \xi \gamma \left[ a(x,\xi) \partial_x^{\gamma+\kappa} \pdd \xi \kappa b(x,\xi) \right] (-1)^{|\kappa|} \\
\end{aligned}
\end{equation}
then
\begin{equation}\label{dj1}
\begin{aligned}
& \pdd x \alpha \pdd \xi \beta d_j(x,\xi)
= (2 \pi i)^{-j} \sum_{|\gamma|+|\kappa|=j} (\gamma! \kappa!)^{-1}
\pdd x \alpha \partial_\xi^{\beta+\gamma}
\left[ a(x,\xi) \partial_x^{\gamma+\kappa} \pdd \xi \kappa b(x,\xi) \right] (-1)^{|\kappa|} \\
= & \ (2 \pi i)^{-j} \sum_{|\gamma|+|\kappa|=j} (\gamma! \kappa!)^{-1}
\sum_{\stackrel{\mu \leq \alpha}{\la \leq \beta+\gamma}}
\binom{\alpha}{\mu} \binom{\beta+\gamma}{\la} \\
& \times
\partial_x^{\alpha-\mu} \partial_\xi^{\beta+\gamma-\la} a(x,\xi)
\partial_x^{\gamma+\mu+\kappa} \partial_\xi^{\la+\kappa} b(x,\xi) (-1)^{|\kappa|}.
\end{aligned}
\end{equation}
Comparison of \eqref{dj1} and \eqref{cj2} gives
\begin{equation}\nonumber
\begin{aligned}
& \pdd x \alpha \pdd \xi \beta \sum_{j<N} \left( c_j(x,\xi)-d_j(x,\xi) \right) \\
& = \sum_{j<N} (2 \pi i)^{-j} \sum_{|\gamma|=j} (\gamma!)^{-1}
\sum_{\stackrel{\mu \leq \alpha}{\la \leq \beta+\gamma}}
\binom{\alpha}{\mu} \binom{\beta+\gamma}{\la}
\partial_x^{\alpha-\mu} \partial_\xi^{\beta+\gamma-\la} a(x,\xi)
F_{\gamma+\mu,\la,N-j} (\xi),
\end{aligned}
\end{equation}
which can be estimated using \eqref{bprime1} and \eqref{symbclass2}. This reveals that $\sum c_j \sim \sum d_j$, and hence $c \sim \sum d_j$.
Thus, using Lemma \ref{formalsumconstruction1} and the sequence $(\varphi_j) \subseteq C_c^\infty(\rr d)$ defined in its proof, we have
\begin{equation}\nonumber
\begin{aligned}
& c(x,\xi) \\
& \sim \sum_{\gamma \geq 0} \sum_{\kappa \geq 0} \varphi_{|\gamma|+|\kappa|}(\xi) (\gamma! \kappa!)^{-1} (2 \pi i)^{-|\kappa+\gamma|} (-1)^{|\kappa|}
\pdd \xi \gamma \left[ a(x,\xi)
\partial_x^{\gamma+\kappa}
\pdd \xi \kappa b(x,\xi) \right] \\
& = \sum_{\gamma \geq 0} \sum_{\kappa \geq 0}
\varphi_{|\gamma|+|\kappa|}(\xi)
\sum_{\la+\sigma=\gamma} \frac{(-1)^{|\kappa|}}{\kappa! \la! \sigma!}  (2 \pi i)^{-|\kappa+\gamma|}
\pdd \xi \la a(x,\xi)
\partial_x^{\gamma+\kappa} \partial_\xi^{\kappa+\sigma} b(x,\xi) \\
& = \sum_{\la \geq 0} (\la!)^{-1} \sum_{\mu \geq 0} \varphi_{|\la|+|\mu|}(\xi) (2 \pi i)^{-|\la+\mu|} \left( \sum_{\kappa+\sigma=\mu} \frac{(-1)^{|\kappa|}}{\kappa!\sigma!} \right)
\pdd \xi \la a(x,\xi)
\partial_x^{\la+\mu} \pdd \xi \mu b(x,\xi) \\
& = \sum_{\la \geq 0} \varphi_{|\la|}(\xi) (\la!)^{-1} (2 \pi i)^{-|\la|}
\pdd \xi \la a(x,\xi) \pdd x \la b(x,\xi),
\end{aligned}
\end{equation}
since
\begin{equation}\nonumber
\begin{aligned}
\sum_{\kappa+\sigma=\mu} \frac{(-1)^{|\kappa|}}{\kappa!\sigma!} = \frac{(e-e)^\mu}{\mu!} = \delta_{|\mu|}
\end{aligned}
\end{equation}
where $e=(1,1,\dots,1)$, and where $\delta_{k}$ is the Kronecker delta, i.e. $\delta_{k}=1$ for $k=0$ and $\delta_{k}=0$ for $k \neq 0$.
Lemma \ref{formalsumconstruction1} now shows that $c \sim a \circ b$, and \eqref{operatorcomposition1}, \eqref{operatorcomposition2} finally give
$$
a(x,D) b(x,D) = a(x,D) \ {^t} b'(x,D) + R_2 = c(x,D) + R_2 + R_3
$$
where $R_2 + R_3$ is regularizing.
\end{proof}

We may now state a regularity result for certain hypoelliptic a.p.~$\Psi DO$s.
We study the following class of hypoelliptic symbols \cite{Hormander1,Rodino1}.

\begin{defn}\label{aphypo1}
A symbol $a \in APS_{\rho,\de}^{m,s}(\rr {2d})$ is called formally $s$-hypoelliptic, denoted $a \in APHS_{\rho,\de}^{m,m_0,s}$ where $m_0 \leq m$, provided there exist $C,C_1,A,B > 0$ such that
\begin{align}
|a(x,\xi)|
\geq & \ C_1 {\eabs \xi}^{m_0}, \quad x \in \rr d, \quad |\xi| \geq A, \label{hypoelliptic1a} \\
\left| \left(  \pdd x \alpha \pdd \xi \beta a (x,\xi) \right) a(x,\xi)^{-1} \right|
\leq & \ C^{|\alpha|+|\beta|} (\alpha!)^{s(\rho-\delta)} \beta! {\eabs \xi}^{-\rho|\beta| + \delta|\alpha|}, \nonumber \\
& x \in \rr d, \quad |\xi| \geq \max(A,B|\beta|^s). \label{hypoelliptic1b}
\end{align}
\end{defn}

\begin{prop}\label{hypoelliptic1}
If $a \in APHS_{\rho,0}^{m,m_0,s}$ then there exists $b \in APS_{\rho,0}^{-m_0,s}$ such that $R_1$ and $R_2$ are regularizing, where
$$
R_1 = a(x,D) b(x,D) - I, \quad R_2 = b(x,D) a(x,D) - I.
$$
\end{prop}
\begin{proof}
As in \cite[Lemma 3.2.27]{Rodino1} we can construct a formal sum $\sum_{j \geq 0} b_j$ such that $\sum_{j \geq 0} b_j \circ a \sim 1$ with $b_0 \in APS_{\rho,0}^{-m_0,s}$. By Lemma \ref{formalsumconstruction1} there exists a symbol $b \in APS_{\rho,0}^{-m_0,s}$ such that $b \sim \sum_{j \geq 0} b_j$. By \cite[Proposition 3.2.18]{Rodino1}, we have $b \circ a \sim 1$.
As in \cite[Lemma 3.2.27]{Rodino1} we also have $b' \in APS_{\rho,0}^{-m_0,s}$ such that $a \circ b' \sim 1$.
It can be verified straightforwardly that the symbol product is associative: $(a \circ b) \circ c \sim a \circ (b \circ c)$, which implies $b' \sim (b \circ a) \circ b' \sim b \circ (a \circ b') \sim b$.
Hence $a \circ b \sim 1$.

Now Proposition \ref{compositioncalculus1} shows that $a(x,D) b(x,D)=c_1(x,D) + R_1$ and $b(x,D) a(x,D)=c_2(x,D) + R_2$, where $R_1$, $R_2$ are regularizing and $c_1 \sim c_2 \sim 1$. By Proposition \ref{zeroregularizing1}, $c_1(x,D)-I$ and $c_2(x,D)-I$ are regularizing, which proves that $a(x,D) b(x,D)-I$ and $b(x,D) a(x,D)-I$ are regularizing.
\end{proof}

As a consequence of Proposition \ref{hypoelliptic1} we get a regularity result for hypoelliptic operators, under the assumption $\delta=0$.
The following result can be seen as a Gevrey space version of Shubin's result \cite[Theorem 5.3]{Shubin1} concerning regularity in a Sobolev space scale for a.p. pseudodifferential operators. See also \cite[Corollary 3.2]{Shubin4}.
If we assume \eqref{frequencyrequirement1} then $G_{\rm ap}^s(\rr d) \subseteq W_{s,0}^1(\rr d)$ by Proposition \ref{frequencyseparation1}.

\begin{cor}\label{hypoellipticregularity1}
If $a \in APHS_{\rho,0}^{m,m_0,s}$, $f \in W_{s,0}^1(\rr d)$ and $a(x,D)f\in G_{\rm ap}^s(\rr d)$ then $f \in G_{\rm ap}^s(\rr d)$.
\end{cor}
\begin{proof}
Proposition \ref{hypoelliptic1} gives $b \in APS_{\rho,0}^{-m_0,s}$ such that $R=b(x,D)a(x,D)-I$ is regularizing. If $g=a(x,D)f$ we have $b(x,D) g = b(x,D) a(x,D)f= f + R f$, and thus $f = b(x,D) g - R f \in G_{\rm ap}^s(\rr d)$, since $R$ is regularizing and $b(x,D) g \in G_{\rm ap}^s(\rr d)$ due to Corollary \ref{gevreycont1}.
\end{proof}

\section{Applications to operators of constant strength}\label{sectionconstant}

Let us recall some definitions from \cite{Hormander0,Hormander1,Rodino1}.
A polynomial $P(\xi)$, $\xi \in \rr d$, is called $s$-hypoelliptic, $1 \leq s < \infty$, if for some $\rho$, $0 < \rho \leq 1$, $\rho \geq 1/s$,
\begin{equation}\label{hypoellipticpolynomial1}
| \pd \beta P(\xi) | \leq C |P(\xi)| (1+|\xi|)^{-\rho |\beta|}, \quad |\xi| \geq A,
\end{equation}
for suitable constants $C>0$, $A \geq 0$.
Moreover, for a given polynomial $P(\xi)$, $\xi \in \rr d$, we set
\begin{equation}\label{strength1}
\wt P(\xi) = \left( \sum_{\alpha \in \nn d} | \pd \alpha P(\xi) |^2 \right)^{1/2}.
\end{equation}
We say that a polynomial $Q(\xi)$ is weaker than $P(\xi)$ if for a suitable $C>0$
\begin{equation}\label{weaker1}
\wt Q(\xi) \leq C \wt P(\xi), \quad \xi \in \rr d.
\end{equation}
If for some $C>0$ we have $C^{-1} \wt P(\xi) \leq \wt Q(\xi) \leq C \wt P(\xi)$ for all $\xi \in \rr d$
we say that the polynomials $Q$ and $P$ are equally strong.
A partial differential operator with symbol $p(x,\xi)$
is said to have constant strength if the polynomials $p(x,\cdot)$ and $p(y,\cdot)$, obtained by ``freezing'' the space variable at two different $x,y \in \rr d$, are equally strong for any $x,y \in \rr d$.

Suppose $P$ is a linear partial differential operator
$$
P = \sum_{|\alpha| \leq m} a_\alpha(x) \left(\frac{D}{2 \pi}\right)^\alpha
$$
of constant strength with coefficients in $a_\alpha \in G_{\rm ap}^s(\rr d)$ and symbol $p(x,\xi) = \sum_{|\alpha| \leq m} a_\alpha(x) \ \xi^\alpha$. Set $P_0(\xi)=p(x_0,\xi)$ for any $x_0 \in \rr d$ and let $P_0(\xi),P_1(\xi)$, $\dots, P_r(\xi)$ be a basis of the finite-dimensional vector space of the polynomials weaker than $P_0(\xi)$.
Then we may write
\begin{equation}\label{operatorfamily1}
P = \sum_{j=0}^r c_j(x) P_j \left( \frac{D}{2 \pi} \right)
\end{equation}
which has symbol
\begin{equation}\label{symbolfamily1}
p(x,\xi) = \sum_{j=0}^r c_j(x) P_j (\xi), \quad x, \ \xi \in \rr d,
\end{equation}
where $c_j \in G_{\rm ap}^s(\rr d)$, $0 \leq j \leq r$.

We now consider a converse situation.
Let $s \geq 1$ and let $P_0(\xi)$ be a fixed $s$-hypoelliptic polynomial of order $m$, and let $P_0(\xi),P_1(\xi)$, $\dots, P_r(\xi)$ be a basis in the finite-dimensional vector space of the polynomials weaker than $P_0(\xi)$.
We consider partial differential operators of the form \eqref{operatorfamily1} with symbol \eqref{symbolfamily1}.
Suppose that $p$ satisfies
\begin{align}
c_j \in G_{\rm ap}^1(\rr d), \quad  j=0,1,\dots, r, \quad \mbox{and} \label{pcondition1} \\
|p(x,\xi)| \geq \ep |P_0 (\xi)|, \quad |\xi| \geq A, \quad x \in \rr d, \label{pcondition2}
\end{align}
for some constants $\ep>0$, $A \geq 0$.

\begin{prop}
If the symbol of the partial differential operator $P$ defined by \eqref{operatorfamily1} satisfies \eqref{pcondition1} and \eqref{pcondition2} then $P$ has constant strength.
\end{prop}
\begin{proof}
Let $x \in \rr d$ be arbitrary and consider $p(x,\cdot)$ as a polynomial.
We have for any $\alpha \in \nn d$
\begin{equation}\nonumber
\begin{aligned}
|\pdd \xi \alpha p(x,\xi)|^2 & \leq \left( \sum_{j=0}^r |c_j(x)| |\pd \alpha P_j(\xi) |\right)^2
\leq (r+1) \sum_{j=0}^r |c_j(x)|^2 |\pd \alpha P_j(\xi) |^2 \\
& \leq C \sum_{j=0}^r |\pd \alpha P_j(\xi) |^2
\end{aligned}
\end{equation}
for some $C>0$ that does not depend on $x$. Hence
\begin{equation}\nonumber
\wt p(x,\xi)^2 = \sum_{\alpha \in \nn d} |\pdd \xi \alpha p(x,\xi)|^2
\leq C \sum_{j=0}^r \wt P_j(\xi)^2
\leq C_1 \wt P_0(\xi)^2, \quad \xi \in \rr d
\end{equation}
for some $C_1>0$, since $P_j$, $j=0,1,\dots,r$, are weaker that $P_0$ by assumption.
Thus $p(x,\cdot)$ is weaker than $P_0$.
On the other hand, the assumption \eqref{pcondition2} gives
\begin{equation}\label{oppositeinequality1}
|P_0 (\xi)| \leq C_2 \wt p(x,\xi), \quad \xi \in \rr d,
\end{equation}
for some $C_2>0$.
In fact, for $|\xi| \geq A$ this follows directly from \eqref{pcondition2}.
For $|\xi|<A$, we observe that
\begin{equation}\label{highestorder1}
C_x := \sum_{|\alpha|=m} |a_\alpha(x)|^2 > 0,
\end{equation}
since otherwise it would result
$|p(x,t \xi)| \leq C t^{m-1}$ for $\xi \in \rr d$ fixed and $t \geq 1$,
which contradicts \eqref{pcondition2} and the assumption that the order of $P_0$ is $m$.
Since $\wt p(x,\xi)^2 \geq C_x$ we may conclude that \eqref{oppositeinequality1} is satisfied also for $|\xi|<A$.

Now \eqref{oppositeinequality1} and \cite[Lemma 3.3.15]{Rodino1} imply that we have $\wt P_0(\xi) \leq C_3 \wt p(x,\xi)$, $\xi \in \rr d$, for $C_3>0$, and hence $P_0$ is weaker than $p(x,\cdot)$. It follows that $p(x,\cdot)$ and $p(y,\cdot)$ are equally strong for all $x,y \in \rr d$.
\end{proof}

In the proof of our final result Theorem \ref{constantstrengthhypoelliptic1} we need the following two lemmas.
Their proofs can be found for example in \cite[Lemma 3.3.14]{Rodino1} and \cite[Lemma 3.3.17]{Rodino1}, respectively.

\begin{lem}\label{finalresultlemma1}
If the polynomial $Q(\xi)$ is $s$-hypoelliptic for some $s \geq 1$, then there exists $C>0$ such that
$$
|Q(\xi)| \leq \wt Q(\xi) \leq C |Q(\xi)|, \quad |\xi|>C.
$$
\end{lem}

\begin{lem}\label{finalresultlemma2}
Let $Q_1,Q_2$ be polynomials.
If $Q_1(\xi)/Q_2(\xi)$ is bounded for large $|\xi|$ and $Q_2(\xi)$ is $s$-hypoelliptic for some $s \geq 1$
then there exists $C>0$ and $A \geq 0$ such that
$$
|\pd \beta Q_1(\xi)| \leq C |Q_2(\xi)| (1+|\xi|)^{-|\beta|/s}, \quad |\xi| \geq A.
$$
\end{lem}

\begin{thm}\label{constantstrengthhypoelliptic1}
Let $P$ be an operator of the form \eqref{operatorfamily1} where \eqref{pcondition1} and \eqref{pcondition2} are satisfied.
If $s \geq 1$, $f \in W_{s,0}^1(\rr d)$ and $P f \in G_{\rm ap}^s(\rr d)$ then $f \in G_{\rm ap}^s(\rr d)$.
\end{thm}
\begin{proof}
It suffices to prove that the symbol $p(x,\xi)$ satisfies conditions \eqref{hypoelliptic1a} and \eqref{hypoelliptic1b} for $\delta=0$.
Corollary \ref{hypoellipticregularity1} will then give the conclusion.

Observe first that the estimates \eqref{hypoellipticpolynomial1} where $0 < \rho \leq 1$, valid for $P_0(\xi)$ by assumption, imply
$$
|P_0(\xi)| \geq \ep_1 \eabs{\xi}^{\rho m}, \quad |\xi| \geq A,
$$
for some $\ep_1>0$, since the order of $P_0$ is $m$ by assumption. Hence from \eqref{pcondition2}
$$
|p(x,\xi)| \geq \ep \ep_1 \eabs{\xi}^{\rho m}, \quad |\xi| \geq A, \quad x \in \rr d,
$$
and \eqref{hypoelliptic1a} is satisfied with $m_0=\rho m$.

On the other hand we have
\begin{equation}\label{symbolestimate3}
| \pdd x \alpha \pdd \xi \beta p(x,\xi)|
\leq \sum_{j=0}^r |\pd \alpha c_j(x)| | \pd \beta P_j(\xi)|.
\end{equation}
Observe that
\begin{equation}\nonumber
|P_j(\xi)| \leq \wt P_j(\xi) \leq C_1 \wt P_0(\xi) \leq C_2 |P_0(\xi)|, \quad |\xi| \geq A, \quad j=0,1,\dots,r,
\end{equation}
$C_1,C_2>0$, $A \geq 0$, because $P_j$ is assumed to be weaker than $P_0$, and we may use Lemma \ref{finalresultlemma1} to obtain the last estimate.
Hence we can apply Lemma \ref{finalresultlemma2} and obtain
\begin{equation}\nonumber
|\pd \beta P_j(\xi)| \leq C |P_0(\xi)| (1+|\xi|)^{-|\beta|/s}, \quad |\xi| \geq A \geq 0, \quad j=0,1,\dots,r.
\end{equation}
Applying \eqref{pcondition1} to estimate $|\pd \alpha c_j(x)|$
and \eqref{pcondition2}, we conclude from \eqref{symbolestimate3}
\begin{equation}\nonumber
| \pdd x \alpha \pdd \xi \beta p(x,\xi)|
\leq C^{1+|\alpha|} \alpha!  |p(x,\xi)| \eabs{\xi}^{-|\beta|/s}, \quad x \in \rr d, \quad |\xi| \geq A,
\end{equation}
for some $C>0$ and $A \geq 0$.
Thus \eqref{hypoelliptic1b} is satisfied with $\rho=1/s$ and $\delta=0$.
The proof is concluded.
\end{proof}

\begin{cor}
Suppose that $P$ is a linear partial differential operator of constant strength with coefficients in $G_{\rm ap}^1(\rr d)$, such that the symbol $p(x,\xi)$ satisfies: \\
\indent \rm{(i)} \ $p(x_0,\cdot)$ is a hypoelliptic polynomial for some $x_0 \in \rr d$, \\
\indent \rm{(ii)} \ $|p(x,\xi)| \geq \ep |p (x_0,\xi)|$ for $|\xi| \geq A \geq 0$, $x \in \rr d$, where $\ep>0$. \\
Then $s \geq 1$, $f \in W_{s,0}^1(\rr d)$ and $P f \in G_{\rm ap}^s(\rr d)$ imply $f \in G_{\rm ap}^s(\rr d)$.
\end{cor}

\begin{example}
Consider the operator
$$
p(x,D) = \sum_{|\alpha| \leq m} c_\alpha (x) \partial^\alpha
$$
with $c_\alpha \in G_{\rm ap}^1(\rr d)$ and assume uniform ellipticity:
$$
|p(x,\xi)| \geq \ep \eabs{\xi}^m, \quad |\xi| \geq A, \quad \ep>0.
$$
The polynomial $P_0=p(0,\cdot)$ is $1$-hypoelliptic (analytic hypoelliptic), and $\{ \xi^\alpha \}_{|\alpha| \leq m}$
is a basis for the polynomials weaker than $P_0$ (see \cite[Theorem 10.4.9]{Hormander1}).
We may then apply Theorem \ref{constantstrengthhypoelliptic1}.
\end{example}

Let us compare with previously existing results.
From \cite[Corollary 3.2]{Shubin4} we know that if $A$ is a uniformly elliptic operator with coefficients in $C_{\rm ap}^\infty(\rr d)$ and $f \in C_{\rm ap}(\rr d)$, then $A f \in C_{\rm ap}^\infty(\rr d)$ implies $f \in C_{\rm ap}^\infty(\rr d)$. Therefore $A f \in G_{\rm ap}^s(\rr d)$ implies $f \in C_{\rm ap}^\infty(\rr d)$.
On the other hand, for $A f \in G_{\rm ap}^s(\rr d)$ we have as well $f \in G^s(\rr d)$, in view of \cite{Hashimoto1,Rodino1,Zanghirati1}, hence $f \in C_{\rm ap}^\infty(\rr d) \cap G^s(\rr d)$.
Besides enlarging the space of admissible solutions to $W_{s,0}^1(\rr d)$, here we can actually conclude $f \in G_{\rm ap}^s(\rr d)$.

\end{document}